\documentclass[11pt]{amsart}

\usepackage{amsbsy, amsmath,amsfonts,amssymb,amsthm,amscd,amssymb,latexsym,}
\usepackage[latin1]{inputenc} 
\usepackage{graphicx}
\usepackage[dvips]{epsfig}
\usepackage{graphics}
\usepackage{psfrag}										 
\usepackage{cancel} 

\usepackage{enumerate} 
\usepackage{indentfirst}
\usepackage{euscript}

 \theoremstyle{plain}
\newtheorem{theorem}{Theorem}
 \newtheorem{lemma}{Lemma}

\newtheorem{proposition}{Proposition}

\usepackage[usenames,dvipsnames]{xcolor}

\usepackage{color}

\def\ron1#1{\textcolor[named]{Magenta}{\large \bf $\mathcal{R}o$: }\textcolor[named]{Plum}{\large \bf {#1}}}

  \title[ Lines of Principal Curvature on Quadrics
  of  $\mathbb  R^4$  ]{    Lines of Curvature on  Quadric
   Hypersurfaces  of  $\mathbb R^4$}
 
  \author{  J. Sotomayor and R. Garcia }

\begin{document}
 
 %
 %
 % \thanks{The first author was partially supported by a doctoral fellowship  CAPES/CNPq.
 % The second author participated in the    FAPESP Thematic Project
 % 2008/02841-4 and has a  fellowship  CAPES PVNS at UNIFEI.  The
 % %second and third authors are fellows of CNPq and participated of
 %%the project CNPq  Proc. 476672/2009-00. The authors were supported
 %%by Pronex/FAPEG/CNPq Proc. 2012 10 26 7000 803.}
 
 %  \keywords{principal curvature lines, ellipsoid, umbilic singularities}
  
 \keywords
  {partially umbilic point, quadrics, principal lines, principal configuration.  }
 
 \subjclass{ 53C12, 57R30, 37C15}
 
 \begin{abstract}
 Here are  described the geometric structures  of the lines of principal curvature 
 and  the  partially umbilic singularities  of
 the tridimensional non compact  generic 
 quadric
 hypersurfaces 
  of  ${\mathbb R}^4$.  This includes  the  {\em ellipsoidal hyperboloids}  of one and two sheets  and the   {\em toroidal hyperboloids}. 
  The present study  complements  the analysis of the compact  {\em ellipsoidal hypersurfaces}    carried out  in \cite{bsbm14}.\\
 \end{abstract}

 \maketitle

\section{Introduction}\label{sec:Intro}
 
The first example of
a {\em principal curvature configuration}  on a surface in  $\mathbb R^3$,  consisting
of
the umbilic points (at which the principal curvatures coincide)
and, outside them, by the  foliations 
with 
 the  minimal and maximal
principal curvature lines  was  determined    
 for the case of the ellipsoid  
   $q(x,y,z)=x^2/a^2+y^2/b^2+z^2/c^2=1, \,  a>b>c>0$. 
See   Monge \cite{mon}  and the illustration in figure  \ref{fig:elip} 
of  this  configuration deduced from Dupin theorem
  \cite{spivak}, \cite{struik}. \\

An extension of the principal configuration for the case of ellipsoidal quadric hypersurfaces in
$\mathbb R^4$ was achieved in \cite{bsbm14}. 
This was preceded by   work of Garcia  \cite{garcia-tese}  where  the generic properties of principal configurations on smooth hypersurfaces were established.   There was determined the principal configuration on  the ellipsoid \\
$\bullet  \; Q_0 
=x^2/a^2+y^2/b^2+t^2/c^2 +w^2/d^2= 1$,
 $a > b > c> d > 0, $ in $\mathbb R^4$.   
It was proved that it carries  four closed  regular curves of  
{\em partially umbilic}  points, along which two of  the three  principal curvatures coincide and   whose transversal structures  are of  the Darbouxian type $D_1$, as at  the umbilic points in the ellipsoid
 with 
 $3$ different axes in ${\mathbb R}^3$. See Fig. \ref{fig:elip}.  
 A different proof of this result 
 was  given in \cite{bsbm14}, 
 where a  complete   description of the  principal configurations 
   on all   tridimensional ellipsoids, including those  with  some coincidence in the  semi-axes $a,b,c,d$ and exhibiting isolated umbilic points.

In this paper will  be determined  the principal configurations for
the  
generic  non compact  
quadric hypersurfaces in 
$\mathbb R^4$, extending in one dimension the classical results  for the 
hyperboloids in  $\mathbb R^3$
 illustrated in figures \ref{fig:elip}, \ref{fig:q1folha} and \ref{fig:q2folhas}.

By a generic non-compact 
quadric hypersurface in   $\mathbb R^4$ is meant the  unit  level hypersurface 
  defined (implicitly) by one of the  non-degenarate 
quadratic forms 
$Q_1$, $Q_2$ and $Q_3$ which, after an orthonormal diagonalization and scalar multiplication, can be written as  follows:\\

 \noindent $\bullet \;  Q_1 
=x^2/a^2+y^2/b^2+z^2/c^2 -t^2/d^2= 1$, $\; a  > b  >  c > 0,\;  d>0$.\;  Ellipsoidal Hyperboloid with one sheet,  presented in section  \ref{ss:q1abcd};
\\
$\bullet \;  Q_2 
=x^2/a^2+y^2/b^2-z^2/c^2 -t^2/d^2= 1$, $\; a  > b  > 0,\;  d>  c >0$.\;Toroidal Hyperboloid, developed  in section \ref{ss:q2abcd};
\\
$\bullet \; Q_3
=x^2/a^2-y^2/b^2- z^2/c^2 - t^2/d^2= 1$ $\; a   > 0,\;  b >  c > d> 0$.\; Ellipsoidal Hyperboloid with two sheets, presented in section \ref{ss:q3abcd}.
\\
The positive orientation will be 
 made explicit in each specific case.

This paper is organized as follows: 
Section \ref{ss:prelim} reviews standard  definitions pertinent to principal configurations on hypersurfaces in $\mathbb R^4$. It also explains the
convention, in colors and print patterns,  used   in 
 their 
illustrations.

Section \ref{ss:q0abcd}
includes a 
a review of 
the main results of \cite{bsbm14} pertinent to the ellipsoid $Q_0$  which  will be  helpful for 
 the present paper.

Section \ref{ss:qr3} 
complements the 
the classical knowledge about 
the principal structures on quadric surfaces in $\mathbb R^3$
with results  not found in standard sources \cite{struik} and \cite{spivak}.

In Sections \ref{ss:q1abcd},  \ref{ss:q2abcd} and  \ref{ss:q3abcd}  are described the principal configurations of   generic quadrics $
Q_1$,  $
Q_2$ and  $
Q_3$.
The last section \ref{ss:CoCo} 
points out to the novelty of the  principal structure
of transversal crossing of partially umbilic separatrix surfaces established in this work.

\section{Preliminaries on Principal Configurations and  Color Conventions  \label{ss:prelim}}

The {\em principal configuration} on an oriented hypersurface  in  $\mathbb R^4$, with euclidean scalar product $\langle \cdot, \cdot \rangle$,   consists on the
 umbilic points, at which  the $3$  principal curvatures coincide, 
 the partially umbilic points, at which only $2$  principal curvatures are equal,   and the integral foliations of the
 three principal line fields on the complement of these sets of points, which will be referred to as the { \em principal regular part} of the  hypersurface.

   Recall that the principal curvatures $k_1 \leq k_2 \leq k_3$
 are the eigenvalues of
 the  automorphism $D(-N)$ of the tangent  bundle of  the hypersurface, 
  taking   $N$  as  the  unit normal  which defines  the positive orientation.

Thus the set of partially umbilic points is the union of $\mathcal P_{12}$,   where $k_1 = k_2  <  k_3$,  and of   $ \mathcal P_{23}$,  where $k_3 = k_2  >  k_1$.
The set $\mathcal U$ of umbilic points is defined by  $k_1 = k_2  =  k_3$.

The eigenspaces corresponding to the  eigenvalues $k_i$  will be denoted by ${\mathcal L}_i, i= 1, 2, 3.$
 They are line fields,
 well defined
 and, for 
 quadrics  also   analytic   on the {\em  principal  regular part} of the hypersurface.
 In fact   ${\mathcal L}_1$ is defined and analytic on the complement of
  $ \mathcal U \cup \mathcal P_{12}$,  ${\mathcal L}_3$ is defined and analytic on the complement of
$ \mathcal U \cup  \mathcal P_{23}$  and ${\mathcal L}_2$ is defined and analytic on the complement of
 $ \mathcal U \cup  \mathcal P_{12} \cup  \mathcal P_{23}$.

A change in the orientation of the hypersurface produces a change of sign in the principal curvatures and an exchange of  ${\mathcal L}_1$  and $P_{12}$ into ${\mathcal L}_3$  and $\mathcal  P_{23}$ and viceversa. The line fields $\mathcal L_2$ is preserved. 

A {\em principal chart} $ ( u_1,\, u_2,\, u_3 )$  is one for which the principal line fields  ${\mathcal L}_i$ are spanned by $\partial/\partial u_i , \, i=1, 2, 3$. Often they will appear as {\em parametrizations} $\varphi$   with   $\partial N /\partial u_i  = -  k_i \partial \varphi/\partial u_i , \, i=1, 2, 3$.  Principal charts are also characterized by the fact that the  first and second  fundamental forms $g_{ij} = \langle  \partial \varphi/\partial u_i  , \partial \varphi/\partial u_i \rangle$ and $b_{ij}   =  \langle  \partial ^2\varphi  /\partial u_i  \partial u_j,  N\rangle$   are simultaneously diagonalized. That is  on principal charts $g_{ij}= b_{ij} = 0, \, i \neq j, $ and $k_i =b_{ii} / g_{ii} , i= 1,2,3$

\subsection {Color and Print Conventions for  Illustrations  in this Paper} \label{ss:CC}

The color convention in  figures 
illustrating principal configurations on quadrics 
in $\mathbb R^4$  is as follows.

\begin{itemize}

\item [] Black ( \rule{0.3cm}{.05cm}    \rule{0.07cm}{.05cm}   { \rule{0.3cm}{.05cm} }  \rule{0.09cm}{.05cm}   { \rule{0.3cm}{.05cm}}  \rule{0.09cm}{.05cm} { \rule{0.3cm}{.05cm}}) integral curves of line field  $\mathcal L_1$,

  \item []Blue (\textcolor{blue}{ \rule{0.3cm}{.05cm}} \textcolor{blue}{ \rule{0.3cm}{.05cm} }\textcolor{blue}{ \rule{0.3cm}{.05cm}} \textcolor{blue}{ \rule{0.3cm}{.05cm}} \textcolor{blue}{ \rule{0.3cm}{.05cm}}) integral curves of line field  $\mathcal L_3$,
 
  \item [] Red (\textcolor{red}{ \rule{1.9cm}{.05cm}}): integral curves of line field $\mathcal L_2$,

  \item [] Green (\textcolor{green}{ \rule{1.9cm}{.06cm}}):  Partially umbilic arcs $\mathcal P_{12}$,

    \item [] Light Blue  (\textcolor{Cyan}{ \rule{1.9cm}{.05cm}}):  Partially umbilic arcs $\mathcal P_{23}$,

 \end{itemize}

The following dictionary has been adopted for illustrations of integral curves appearing in  figures
 \ref{fig:pabcd},    \ref{fig:conexao},   \ref{fig:q1folha},   \ref{fig:q2folhas},  \ref{fig:Q2global},  \ %ref{fig:pumbq3},
   \ref{fig:PC_q3_123} of 
    this paper when printed in black and white:
  dashed, for blue; dotted-dashed-dotted, for black; full trace, for red.

\section{Ellipsoid with four different axes  in $\mathbb R^4 $ \cite{bsbm14}} \label{ss:q0abcd}

\begin{lemma}\label{lem:cartaprincipale}
The ellipsoid 
\begin{equation} \label{eq:Q0}
Q_0=\dfrac{x^2}{a^2 } +\dfrac{ y^2}{b^2 } +\dfrac{
z^2}{c^2 } + \dfrac{t^2}{d^2 } - 1=0, \;\;\; a>b>c>d>0,
\end{equation}
has
sixteen  principal charts $(u,v,w)=\varphi_i(x,y,z,t)$, $ 0>u>v>w$, where
$$\varphi_i^{-1}:  (-c^2,-d^2)  \times (-b^2,-c^2) \times (-a^2,-b^2)\to 
\{(x,y,z,t): xyzt\ne 0\}\cap Q_0^{-1}(0) $$
 is defined by equations 
\eqref{eq:pchart1}.

\begin{equation}\label{eq:pchart1} \aligned x^2 &= \dfrac{a^2(a^2 + u)(a^2 + v)(a^2 +w )}{(a^2 - b^2)(a^2 - c^2)(a^2 -
d^2)},
\; y^2 =  \dfrac{b^2(b^2 +u)(b^2 + v)(b^2 + w )}{(b^2 - a^2)(b^2 - c^2)(b^2 - d^2)},\\
z^2 &=\dfrac{ c^2(c^2 +u)(c^2 + v)(c^2 +w )}{(c^2 - a^2)(c^2 - b^2)(c^2 - d^2)},\;
  w^2 =  \dfrac{d^2(d^2 +u)(d^2 + v)(d^2 +w )}{(d^2 - a^2)(d^2 - b^2)(d^2 - c^2)}.
\endaligned \end{equation}

For all $p\in \{(x,y,z,t): xyzt \ne
0\}\cap Q_0^{-1}(0)$ and $ Q_0=Q_0^{-1}(0)$ positively oriented
 by the inward normal directed by 
$-\nabla Q_0$, the principal curvatures satisfy
$0<k_1(p)<k_2(p)<k_3(p)$.

The hyperplanes $x=0$ and $z=0$ are invariant by the foliations $\mathcal F_1 $ and $\mathcal F_2$, while  the hyperplanes  $y=0$ and $t=0$ are invariant by the foliations $\mathcal F_2 $ and $\mathcal F_3.$

\end{lemma}

The first fundamental form is given by:
{\small  
\begin{equation} \aligned \label{Ifelip}  I=& \frac 14  \,\frac { \left( u-w \right)  \left( u-v \right) u}{\xi_0(u)} du^2+ \frac 14\, \frac { \left( v-w \right)  \left( u-v \right) (-v)}{  \xi_0(v)}
  dv^2\\
   +& \frac 14\,{\frac { \left( v-w \right)  \left( u-w \right) w}{ \xi_0(w)}}
  dw^2 ,\\
  \xi_0(x)=& \left( {d}^{2
  }+x \right)  \left( {c}^{2}+x \right)  \left( {b}^{2}+x \right) 
   \left( {a}^{2}+x \right).  \endaligned
 \end{equation}
 }
The second fundamental form is given by:

{\small  
\begin{equation} \aligned \label{IIfelip}  II=& -\frac 14 \,{\frac {bcda \left( u-w \right)  \left( u-v \right) }{\sqrt {-uv
w}\;\xi_0(u) }}
 du^2+\frac 14 \,{\frac {dcba \left( v-w \right)  \left( u-v \right) }{\sqrt {-uvw
 }\; \xi_0(v) }}
  dv^2\\
   - &\frac 14 ,{\frac { \left( u-w \right)  \left( v-w \right) abcd}{\sqrt {-uv
   w} \;\xi_0(w) }}
  dw^2 .   \endaligned
 \end{equation}
 }
Therefore,

\begin{equation}\label{eq:kipelip}
k_1=-{\frac {abcd}{u \sqrt {-uvw} }},\;\;\;\;\; k_2=-{\frac {abcd}{v\sqrt {-uvw}}},\;\;\;\; \;k_3=-{\frac {abcd}{w\sqrt {-uvw}}}.\;\;\;\; 
\end{equation}

A theorem proved  in \cite{bsbm14} is reviewed now.

\begin{theorem}  \cite{bsbm14}\label{th:eabcd}
The umbilic set  of the ellipsoid 
  $ 
 Q_0
 $ in equation (\ref{eq:Q0})
 is empty and its partially umbilic set  consists of four closed curves
${\mathcal P}_{12}^1, \; {\mathcal P}_{12}^2,\; {\mathcal P}_{23}^1$ and $ {\mathcal P}_{23}^3$,
 which 
in the chart $(u,v,w)$  defined by 
$$\alpha_\pm (u,v,w)=(au,bv,cw,\pm d\sqrt{1-u^2-v^2-w^2})$$
are given by:  
\begin{equation} \label{eq:elip_hyp_a}
\aligned w=&0,\;{\frac {{u}^{2} \left( {a}^{2}-{d}^{2} \right) }{{a}^{2}-{c}^{2}}}+{
\frac {{v}^{2} \left( {b}^{2}-{d}^{2} \right) }{{b}^{2}-{c}^{2}}}=1
, \;\; \text{ellipse},\\
v=&0,\;  {\frac {{u}^{2} \left( {a}^{2}-{d}^{2} \right) }{{a}^{2}-{b}^{2}}}-{
\frac {{w}^{2} \left( {c}^{2}-{d}^{2} \right) }{{b}^{2}-{c}^{2}}}=1
,\; \text{hyperbole}.\endaligned
\end{equation}

  \noindent i)  \; The principal foliation  ${\mathcal F}_1  $
is singular on ${\mathcal P}_{12}^1\cup  {\mathcal P}_{12}^2,\;$ ${\mathcal F}_3 $
is singular on ${\mathcal P}_{23}^1\cup  {\mathcal P}_{23}^2$  and  ${\mathcal F}_2  $
is singular on ${\mathcal P}_{12}^1\cup  {\mathcal P}_{12}^2 \cup {\mathcal P}_{23}^1\cup  {\mathcal P}_{23}^2$.
 
\noindent  The partially umbilic curves ${\mathcal P}_{12}^1\cup  {\mathcal P}_{12}^2 $
 whose transversal structures are 
of type  $D_1$, have  their  partially  umbilic  separatrix surfaces 
  spanning  a cylinder  $ C_{ 12} $ such that
$\partial C_{12}$ $ =  {\mathcal P}_{12}^1\cup  {\mathcal P}_{12}^2$.

Also, the partially umbilic curves ${\mathcal P}_{23}^1$ and $  {\mathcal P}_{23}^2$
are of type  $D_1$, 
their  partially umbilic separatrix surfaces   span a cylinder  $ C_{23} $ so that $\partial C_{23}=  {\mathcal P}_{23}^1\cup  {\mathcal P}_{23}^2$.

\noindent  All the leaves of    ${\mathcal F}_1  $   outside the cylinder
 $  C_{12}$ are 
 diffeomorphic to  $\mathbb S^1.\;$
 Analogously for the principal foliation ${\mathcal F}_3$ outside the cylinder $C_{23}$. 
 See 
 figure  \ref{fig:pabcd}.
 
    \noindent ii)  \;  The principal foliation ${\mathcal F}_2   $ is singular at
${\mathcal P}_{12}^1\cup  {\mathcal P}_{12}^2 \cup {\mathcal P}_{23}^1\cup  {\mathcal P}_{23}^2$  and there are Hopf bands (cylinders with boundaries consisting  on  two linked closed curves)
 $ H_{123}^1$ and $ H_{123}^2 $ such that $\partial H_{123}^1 ={\mathcal P}_{12}^1\cup {\mathcal P}_{23}^1\;$ and $\partial H_{123}^2 ={\mathcal P}_{12}^2\cup {\mathcal P}_{23}^2\;$, which are partially  umbilic 
 separatrix surfaces. 
 
\noindent All the leaves of ${\mathcal F}_2$,  outside the partially umbilic  separatrix surfaces, are  
diffeomorphic to  $\mathbb S^1.\;$
 See
 figure  \ref{fig:conexao}.  
\end{theorem}

\begin{figure}[h]
\begin{center}
    \def\svgwidth{0.60\textwidth}
    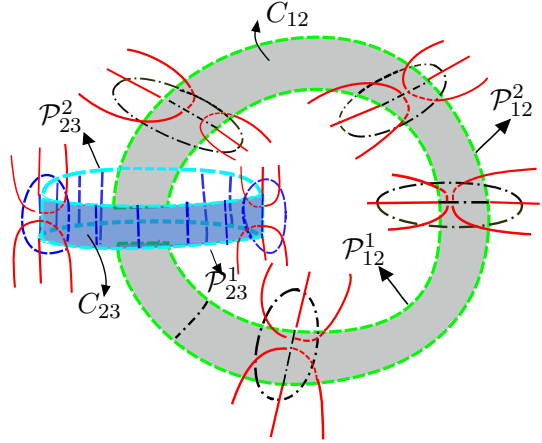
    \caption{{\small  Global behavior of the principal foliations  $\mathcal{F}_i,\; (i=1,2,3)$. The cylinder $C_{12}$ is foliated by principal lines of ${\mathcal F}_1$  and is  bounded  by two partially umbilic lines (dashed  green). The cylinder $C_{23}$ is foliated by principal lines of $\mathcal F_3$ and is  bounded  by  two partially umbilic lines (dashed light blue). }
     }
  \label{fig:pabcd}
    \end{center}
\end{figure}

\begin{figure}[h]
\begin{center}
         \def\svgwidth{0.70\textwidth}
    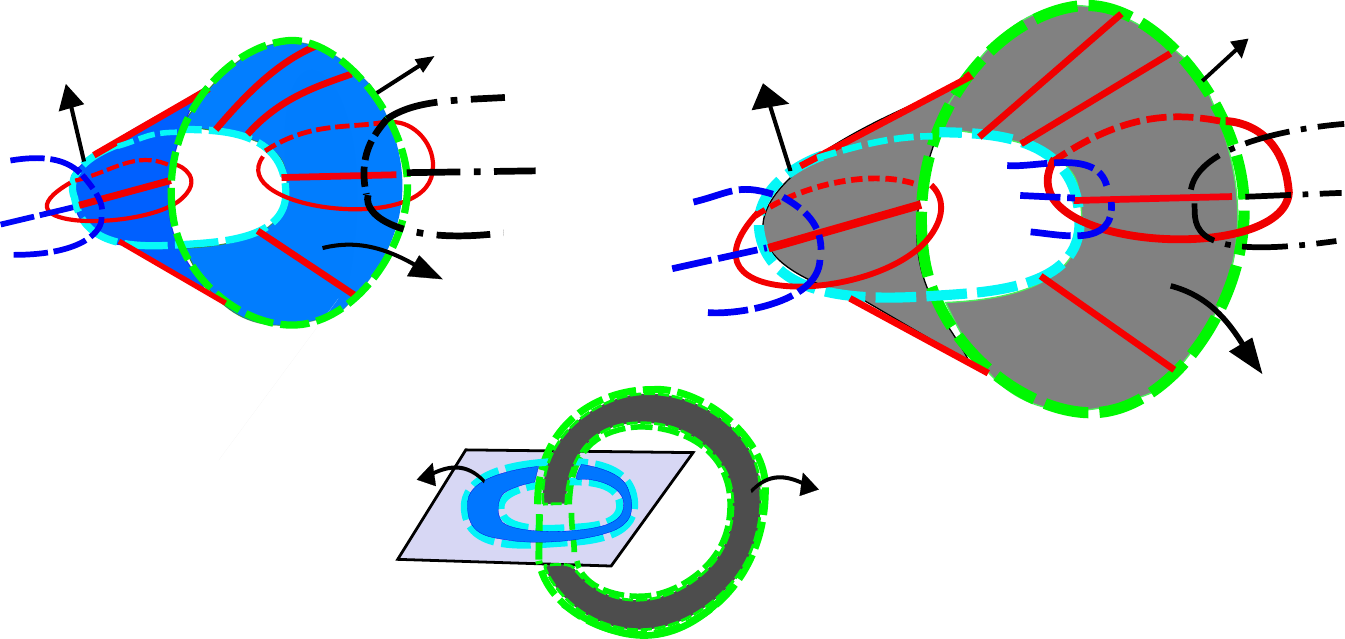
        \includegraphics[scale=0.6]{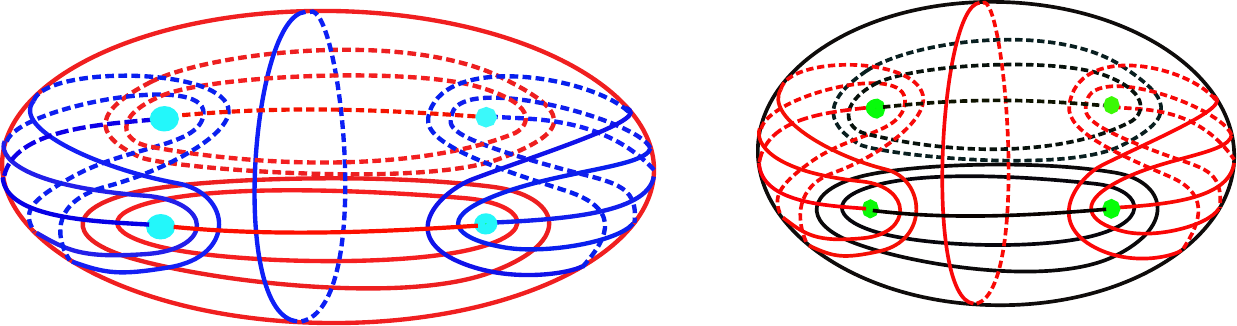}
    \caption{{\small    Hopf bands $H_{123}^1$ and  $H_{123}^2$ with  leaves of $\mathcal{F}_2 $.
    Global behavior of the Principal Foliations  $\mathcal{F}_i  $ near 
    the partially umbilic curves, top. Bottom:  quartic
     ellipsoidal  surface whose $4$ umbilics  slide along the  partially umbilic  closed lines  (horizontal, blue dotted print,
    and vertical, green dotted print).}
 }
  \label{fig:conexao}
    \end{center} 
\end{figure}
 
 \newpage
\section{
Complements on Principal Structures on  Quadrics 
in
 $\mathbb R^3$}\label{ss:qr3}

In this section a theorem of Dupin
 is revisited. 
No 
presentation 
as explicit as the one provided here has been found in the literature. See 
\cite{spivak,  struik}. 

\begin{proposition} \label{prop:qelip} The 
ellipsoid 
\begin{equation} \label{eq:q0}
q_0(x,y,z)=\frac{x^2}{a^2}+\frac{y^2}{b^2}+\frac{z^2}{c^2}-1=0, \; a>b> c>0, 
\end{equation}
is diffeomorphic to $\mathbb S^2 $ and 
 has
  eight principal charts $(u,v)=\varphi_i(x,y,z)$, $u>v$,  where
  $\varphi_i^{-1}:   (-a^2,-b^2 )     \times (-b^2,-c^2) \to 
  \{(x,y,z): xyz\ne 0\}\cap q_0^{-1}(0) $
   is defined by equations
  \eqref{eq:elip1}.

  \begin{equation}\label{eq:elip1} \aligned x^2 &= {\frac {a^2   \left( {a}^{2}+v \right)  \left( {a
  }^{2}+u \right) }{    \left( {a}^
  {2}-{c}^{2} \right)  \left( a^2-{b}^{2} \right)}},
  \;  y^2 = -{\frac {{b}^{2}   \left( {b}^{2}+v \right) 
   \left( {b}^{2}+u \right) }{   \left( {
  a}^{2}-{b}^{2} \right)  \left( {b}^{2}-{c}^{2} \right) }}
  \\
  z^2 &=-{\frac {{c}^{2}   \left( {c}^{2}+v \right) 
   \left( {c}^{2}+u \right) }{ \left( a^2- {c}^{2}\right)  \left( {b
  }^{2}-{c}^{2} \right)    }}
  \endaligned \end{equation}

  For all $p\in \{(x,y,z): xyz \ne
  0\}\cap q_0^{-1}(0)$ and $q_0=q_0^{-1}(0)$ 
  positively oriented
  by the inward normal,
  the principal curvatures satisfy
  $k_1(p)\leq   k_2(p)$ and in the chart $(u,v)$ are given by:
  %
%  \begin{equation}
 $ k_1(u,v)={ -\frac {abc}{u\sqrt {uv}}},\quad k_2(u,v)={-\frac {abc}{v\sqrt {uv}}}. $

     The   umbilic set is given by  \;   
     $${\mathcal U}=\{p: k_1(p)=k_2(p)\}={\small  \left( 
   \pm \sqrt {{\frac {{a}^{2} \left( {a}^{2}-{b}^{2} \right) }{{a}^{2}-{c}^{
   2}}}},0,\pm \sqrt {{\frac {{c}^{2} \left( {b}^{2}-{c}^{2} \right) }{{a}^{2
   }-{c}^{2}}}}   
   \right)}. $$
   The leaves of the principal foliations ${\mathcal F}_1$  and  ${\mathcal F}_2$ are all  closed 
   (with the exception of  the  four umbilic separatrix arcs). See figure
   %Fig. \cite[page  82]{struik}.
 \ref{fig:elip}.
   
   \begin{figure}[h]
    
   \begin{center}
    \def\svgwidth{0.70\textwidth}
     \includegraphics[scale=0.420]{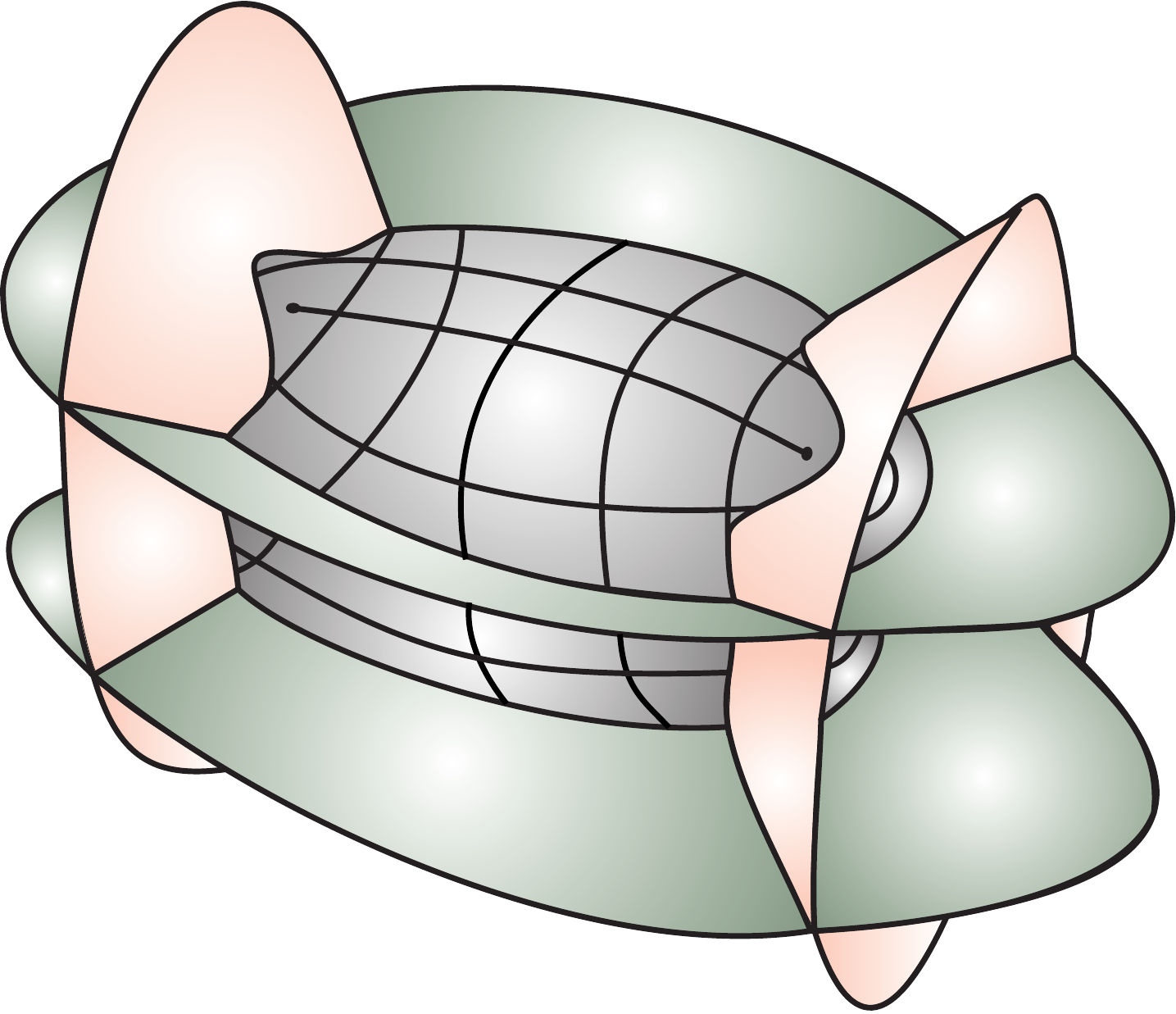}
   \caption{Illustration of  Dupin Theorem. Principal lines of  the ellipsoid  $
   Q_0$ and orthogonal family of quadrics in $\mathbb R^3$. \label{fig:elip} }
   \end{center}
   \end{figure}
\end{proposition}

The proof follows from lemma \ref{lem:q0r3conforme}.

   \begin{lemma}\label{lem:q0r3conforme}
   Consider the ellipsoid   $q_0$ 
   given  in $\mathbb R^3$ by equation (\ref{eq:q0}) with $a>b> c>0$.
 Let
 {\small
 $$ s_1 = \frac 12 \int_{-b^2}^{-c^2} \sqrt{\frac{ u}{( a^2+u ) ( c^2+u )}}  du < \infty \;  
s_2=\frac 12 \int_{-a^2}^{-b^2}   \sqrt{\frac{     v }{ ( a^2+v ) ( c^2+v ) }} dv < \infty.$$
}
   There exists a parametrization
   $\varphi: [-s_1,s_1]\times  [-s_2,s_2]\to 
   q_0  \cap \{(x,y,z), y\geq 0)\}$ such that
 the principal lines are the coordinate
 curves and $\varphi$ is conformal in the interior of the rectangle.

 Moreover $ \varphi(s_1,s_2)=U_1,\;   \varphi(-s_1,s_2)=U_2,\; \varphi(-s_1,-s_2)=U_3$  and $  \varphi(s_1,-s_2)=U_4. $

By symmetry considerations the same result
 holds  for   the region
$q_0\cap \{(x,y,z), y\leq 0)\}$.

  \end{lemma}

  \begin{proof} The ellipsoid
 $q_0
 $
  has a principal chart   $(u,v)$ defined by the parametrization
  $\psi:
   [-b^2,-c^2]\times[-a^2,-b^2]\to \{(x,y,z): x>0,\;y>0,\; z>0\}$ given by:
   {\small  
  $$\psi(u,v)= \left(a\sqrt{\frac{(a^2+v)(a^2+u)}{(a^2-c^2)(a^2-b^2)}},  b\sqrt{\frac{(b^2+v)(b^2+u)}{(a^2-b^2)(c^2-b^2)}},c \sqrt{\frac{(c^2+v)(c^2+u)}{(b^2-c^2) (a^2-c^2)}}\right
  ).$$
  }
  The fundamental forms
   in this chart  are given by:
  $$\aligned I=&\; Edu^2+Gdv^2=\frac 14\frac{u(u-v)}{h(u)}du^2-\frac 14\frac{v(u-v)}{h(v)}dv^2,\\
  II=& \;edu^2+gdv^2= \frac{abc(u-v)}{4\sqrt{uv} h(u)}du^2-\frac{abc(u-v)}{4\sqrt{uv} h(v)}du^2\\
  h(t)=&( a^2+t )( b^2+t )( c^2+t ).
  \endaligned $$
  The principal curvatures are given by  $k_2(u,v)=-\frac{abc}{v\sqrt{uv}},\; k_1(u,v) = -\frac{abc}{u\sqrt{uv}}.$
   Therefore, $k_1(u,v)=k_2(u,v)$ if and only if $u=v=-b^2$.

Considering the change of coordinates defined by
$ds_1=\sqrt{\frac{u}{2h(u)}}du$, $ds_2=\sqrt{\frac{v}{2h(v)}}dv$, 
written 
 $u=A(s_1)$ and $v=B(s_2)$,
 obtain a conformal
parametrization $\varphi:[0,s_1]\times [0,s_2]\to   \{(x,y,z):
x>0,\;y>0,\; z>0\}$
 in which
 the coordinate curves are principal
lines and the fundamental forms
 are given by $I=(A(s_1)-B(s_2))(ds_1^2+ds_2^2)$ and $II=(A(s_1)-B(s_2)) ( k_1ds_1^2+k_2ds_2^2)$.  
 
 From  the symmetry of the ellipsoid $ q_0 $
 with respect to coordinate plane reflections,  consider an analytic continuation  of $\varphi$
 from  the rectangle $R=[-s_1,s_1]\times [-s_2,s_2]$ 
 to obtain a conformal chart $(U,V)$  of  $R$
covering  the region $
q_0
\cap \{y\geq 0\}$.

 By construction $\varphi(\partial R)$ is the ellipse in
 the plane $xz$ and  the four vertices of the rectangle $[-s_1,s_1]\times [-s_2,s_2]$
 are mapped by $\varphi$  to the four umbilic points 
 given by
 $\left(\pm a\sqrt{\frac{a^2-b^2}{a^2-c^2}}, 0, \pm c \sqrt{\frac{b^2-c^2}{a^2-c^2}}\right)$.
 
 An explicit parametrization  $\varphi:[-s_1,s_1]\times [-s_2,s_2] \to \mathbb R^3$ is given by
 {\small  
 $$\aligned \varphi(u,v)=&\left(  a \cos U \sqrt {{ A_1}\,   \cos^2
  V  +   \sin^2V} ,b\sin U \sin V ,c\cos  V \sqrt {  B_1    \cos^2 U
 +  \sin^2 U}
 \right),\\
 A_1=&\frac{a^2-b^2}{a^2-c^2},\;\;B_1=\frac{b^2-c^2}{a^2-c^2} \\
  U=& -s_1+\frac{2s_1}{\pi} u, V=-s_2+\frac{2s_2}{\pi} v, \; u\in [0, \pi], v\in [0,\pi].
 \endaligned
 $$
 }

 See figure  \ref{fig:cq}, left.
 
  \end{proof}

\begin{proposition} \label{prop:q1folha} The quadric $q_1(x,y,z)=\frac{x^2}{a^2}+\frac{y^2}{b^2}-\frac{z^2}{c^2}=1, \; a>b>0, \; c>0, $
 is diffeomorphic to $\mathbb S^1\times \mathbb R$ and 
 has eight principal charts $(u,v)=\varphi_i(x,y,z)$, $u>v$,  where
  $$\varphi_i^{-1}:   (c^2,\infty)     \times (-a^2,-b^2) \to 
  \{(x,y,z): xyz\ne 0\}\cap q_1^{-1}(0) $$
   is defined by equation
  \eqref{eq:q1folha}.

  \begin{equation}\label{eq:q1folha} \aligned x^2 &= {\frac {a^2   \left( {a}^{2}+v \right)  \left( {a
  }^{2}+u \right) }{    \left( {a}^
  {2}+{c}^{2} \right)  \left( a^2-{b}^{2} \right), }}
  \quad y^2 = -{\frac {{b}^{2}   \left( {b}^{2}+v \right) 
   \left( {b}^{2}+u \right) }{   \left( {
  a}^{2}-{b}^{2} \right)  \left( {b}^{2}-{c}^{2} \right) }}
  \\
  z^2 &=-{\frac {{c}^{2}   \left( {c}^{2}-v \right) 
   \left( {c}^{2}-u \right) }{ \left( a^2+ {c}^{2}\right)  \left( {b
  }^{2}+{c}^{2} \right)    }}
  \endaligned \end{equation}

  For all $p\in \{(x,y,z): xyz \ne
  0\}\cap q_1^{-1}(0)$ and $q_1=q_1^{-1}(0)$ positively oriented outward, by  $-\nabla q_1$, 
  the principal curvatures satisfy
  $k_1(p)< 0<  k_2(p)$ and in the chart $(u,v)$ are given by:
  
  \begin{equation}
  k_1(u,v)={ -\frac {abc}{u\sqrt {-uv}}},\quad k_2(u,v)={\frac {abc}{v\sqrt {-uv}}},\ 
  \end{equation}
  
   The   umbilic set   is empty: 
   ${\mathcal U}=\{p: k_1(p)=k_2(p)\}=\emptyset$.
   
   The leaves of the principal foliation ${\mathcal F}_1$ are all  closed   curves  and
   those  of  ${\mathcal F}_2$ are all  open  arcs.
   
    There exists a parametrization
      $\varphi:  \mathbb R\times [-s_2,s_2]\times  \to Q_1\cap\{(x,y,z):  x>0\}$  such that
    the principal lines are the coordinate
    curves and $\varphi$ is conformal in the interior of the domain.

    \begin{figure}[ht]
      \begin{center}
       \def\svgwidth{0.60\textwidth}
       \includegraphics[scale=0.300]{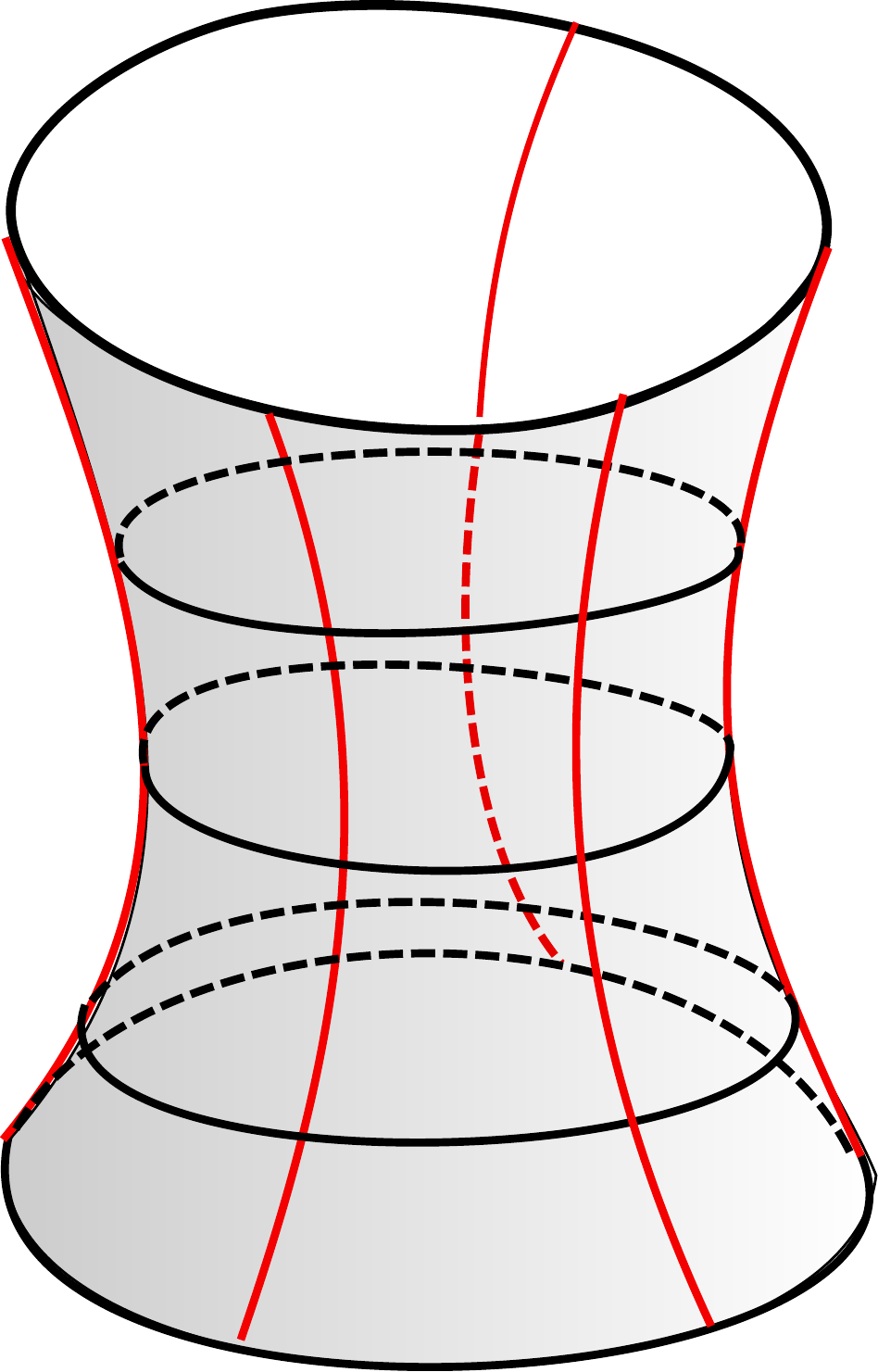}
      \caption{Principal lines of   the quadric $q_1$:\; Hyperboloid with one sheet  in $\mathbb R^3$. \label{fig:q1folha} }
      \end{center}
      \end{figure}
\end{proposition}

\begin{proof} The parametrizations stated are obtained showing that the quadric $q_1$ belongs to a triple orthogonal systems of surfaces defined by

$$q(\lambda,x,y,z)=\frac{x^2}{a^2+\lambda}+\frac{y^2}{b^2+\lambda}-\frac{z^2}{c^2-\lambda}-1=0.$$

Solving the system $q(u,x,y,z)=q(v,x,y,z)=q(0,x,y,z)=0$ in the variables $x^2, \; y^2$ and $z^2$ the equation \eqref{eq:q1folha} is obtained.

From the parametrization $\varphi(u,v) =(x(u,v), y(u,v),z(u,v)) $  defined by equation \eqref{eq:q1folha} and taking values in the positive orthant and considering the   positive unitary normal proportional to 
 $N=(-\frac{x}{a^2},-\frac{y}{b^2},\frac{z}{c^2})=-\nabla q_1$ it follows that the first and second fundamental forms are given by:

$$\aligned I=& Edu^2+Gdv^2= \frac 14 (u-v)\left[\,- \frac{  u}{h(u)}du^2+ \frac{ v}{ h(v)}\right]
\\
II=& edu^2+gdv^2= \frac 14\,{\frac {abc \left( u-v \right)  }{\sqrt {-uv}}}\left[ \frac{du^2}{h(u)}-\frac{dv^2}{h(v)}\right]
\\
h(x)=& \left( {c}^{2}-x \right) 
 \left( {b}^{2}+x \right)  \left( {a}^{2}+x \right) ,\;   c^2 <u, \;  -a^2<v<-b^2<0.
\endaligned $$
Therefore, $k_1=-{\frac {abc}{u\sqrt {-uv} }}<0$ and $0<k_2=-{\frac {abc}{v\sqrt {-uv} }}.$

Let $d\tau_1= \sqrt{-\frac{u}{2h(u)}}du $ and $d\tau_2=\sqrt{\frac{v}{2h(v)}}dv $.  Defining the change of coordinates
$\phi(u,v)=(U,V)$,
  $ U =\int_{c^2}^u  d\tau_1,$
$ V =\int_{- a^2}^v d\tau_2$
   it follows $\beta:[0,\infty)\times [0,s_2]\to q_1$ defined by
 $\beta(U, V ) =\varphi\circ \phi^{-1}(U,V)  $ is a conformal parametrization of the hyperboloid $Q_1$  in the region $\{(x,y,z): x>0, y>0, z>0\}.$
 By symmetry $\beta$ can be extended to the domain $(-\infty,\infty)\times [-s_2,s_2]$, where $s_2=\int_{-a^2}^{-b^2}d\tau_2,$ and taking values in the $Q_1\cap \{(x,y,z): x>0\}.$
\end{proof}

\begin{proposition} \label{prop:q2folhas}The quadric $q_2(x,y,z)=\frac{x^2}{a^2}-\frac{y^2}{b^2}-\frac{z^2}{c^2}=1, \; a>0,\; b>c>0, $   has two connected components, each one diffeomorphic to $ \mathbb R^2$ and 
 has
  eight principal charts $(u,v)=\varphi_i(x,y,z,)$, $u>v$,  where
  $$\varphi_i^{-1}:   (b^2,\infty)     \times (c^2,b^2) \to 
  \{(x,y,z): xyz\ne 0\}\cap q_2^{-1}(0) $$
   is defined by equation
  \eqref{eq:q2folhas}.

  \begin{equation}\label{eq:q2folhas} \aligned x^2 &= {\frac {a^2   \left( {a}^{2}+v \right)  \left( {a
  }^{2}+u \right) }{    \left( {a}^
  {2}+{c}^{2} \right)  \left( a^2+{b}^{2} \right) }}
  \quad y^2 =  {\frac {{b}^{2}   \left( {b}^{2}-v \right) 
   \left( {b}^{2}-u \right) }{   \left( {
  a}^{2}+{b}^{2} \right)  \left( {b}^{2}-{c}^{2} \right) }}
  \\
  z^2 &={\frac {{c}^{2}   \left( {c}^{2}-v \right) 
   \left( {c}^{2}-u \right) }{ \left( a^2+ {c}^{2}\right)  \left( {b
  }^{2}-{c}^{2} \right)    }}
  \endaligned \end{equation}

  For all $p\in \{(x,y,z): xyz \ne
  0\}\cap q_2^{-1}(0)$ and $q_2=q_2^{-1}(0)$ 
  positively oriented
 by the inward normal,
 the principal curvatures satisfy
  $k_1(p)\leq  k_2(p)$ and in the chart $(u,v)$ are given by:
  
  \begin{equation}
  k_1(u,v)={\frac {abc}{u\sqrt {uv}}},\quad k_2(u,v)={\frac {abc}{v\sqrt {uv}}},\ 
  \end{equation}
  
   The   umbilic set\,   ${\mathcal U} %=\{p: k_1(p)=k_2(p)\}
   = \{  \left( \pm\sqrt {{\frac {{a}^{2} \left( {a}^{2}+{b}^{2} \right) }{{c}^{2}+{a}^{
   2}}}},0,\pm\sqrt {{\frac {{c}^{2} \left( {b}^{2}-{c}^{2} \right) }{{c}^{2
   }+{a}^{2}}}}\right)
   \}$.
   
   The leaves of the principal foliation ${\mathcal F}_1$,  
    with the exception of the two umbilic separatrix arcs, 
   are all  closed 
   curves.    All the leaves of  ${\mathcal F}_2$ are   open 
  arcs and it has four umbilic separatrix arcs.
  
   There exists a parametrization
        $\varphi:  \mathbb R\times [-s_2,s_2]\times  \to q_2\cap\{(x,y,z):  x>0, z>0\}$  such that
      the principal lines are the coordinate
      curves and $\varphi$ is conformal in the interior of the domain.
      
   \begin{figure}[ht]
   \begin{center}
    \def\svgwidth{0.90\textwidth}
    \includegraphics[scale=0.400]{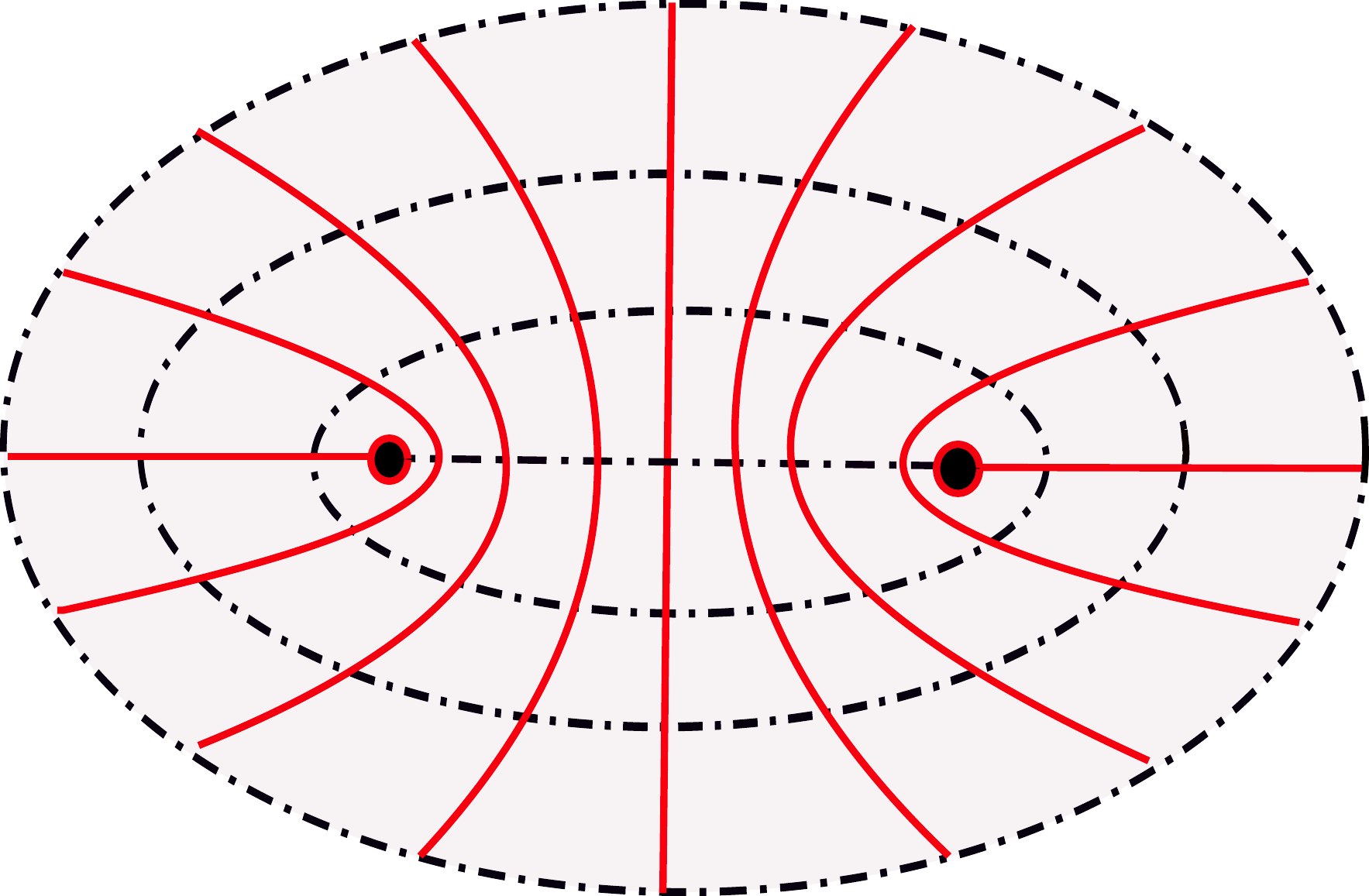}
   \caption{Principal lines on  the quadric $q_2$:\; Hyperboloid with two  sheets  in $\mathbb R^3$. Illustration  in a 
   sheet represented in an open disk. 
    \label{fig:q2folhas} }
   \end{center}
   \end{figure}
\end{proposition}

 \begin{proof} Similar to the proof  of  Proposition     \ref{prop:q1folha}.\end{proof}

   \begin{figure}[ht]
   \begin{center}
    \def\svgwidth{1.00\textwidth}
       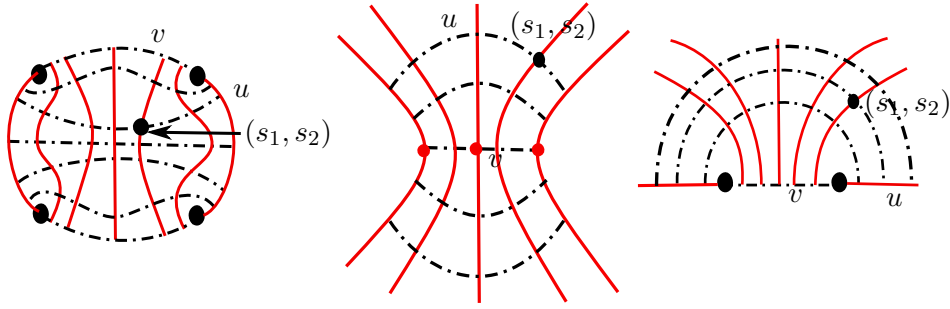
   \caption{ Principal coordinates on quadrics in $\mathbb R^3$. Ellipsoid: left; Hyperboloid with one sheet (center) and with two sheets (right).
    \label{fig:cq} }
   \end{center}
   \end{figure}

 \section{
 Ellipsoidal Hyperboloid with one sheet in $\mathbb R^4$
  } \label{ss:q1abcd}

  \vskip .3cm
  
  \noindent In this section will be established  the global behavior
  of principal lines in the quadric of signature 1 $(++ + - ) $. 
 
 \begin{proposition}\label{prop:AQ1R4}
 The quadric  
 
 \begin{equation} \label{eq:Q1}
 Q_1(x,y,z,t)=\dfrac{x^2}{a^2 } +\dfrac{ y^2}{b^2 } +\dfrac{
 z^2}{c^2 } - \dfrac{t^2}{d^2 } - 1=0, \;\;\; a>b>c>0, \; d>0,
 \end{equation}
 is diffeomorphic to ${\mathbb S}^2\times \mathbb R$ and has
 sixteen  principal charts $(u,v,w)=\varphi_i(x,y,z,t)$, $u>  0>v>w$,  where
 $$\varphi_i^{-1}:   (d^2,\infty)   \times (-b^2,-c^2)   \times (-a^2,-b^2) \to 
 \{(x,y,z,t): xyzt\ne 0\}\cap Q_1^{-1}(0) $$
  is defined by equation
 \eqref{eq:AQ1R4}.

 \begin{equation}\label{eq:AQ1R4} \aligned x^2 &= {\frac {a^2 \left( {a}^{2}+w \right)  \left( {a}^{2}+v \right)  \left( {a
 }^{2}+u \right) }{ \left( {a}^{2}+{d}^{2} \right)  \left( {a}^
 {2}+{c}^{2} \right)  \left( a^2-{b}^{2} \right) }}
 \; y^2 = -{\frac {{b}^{2} \left( {b}^{2}+w \right)  \left( {b}^{2}+v \right) 
  \left( {b}^{2}+u \right) }{ \left( {b}^{2}+{d}^{2} \right)  \left( {
 a}^{2}-{b}^{2} \right)  \left( {b}^{2}-{c}^{2} \right) }}
 \\
 z^2 &={\frac {{c}^{2} \left( {c}^{2}+w \right)  \left( {c}^{2}+v \right) 
  \left( {c}^{2}+u \right) }{ \left( a^2- {c}^{2}\right)  \left( {b
 }^{2}-{c}^{2} \right)  \left(d^2 +{c}^{2}\right) }}
 ,\;
  t^2 =-{\frac {{d}^{2} \left( {d}^{2}-w \right)  \left( {d}^{2}-v \right) 
    \left( {d}^{2}-u \right) }{ \left( {d}^{2}+{a}^{2} \right)  \left( {d
   }^{2}+{b}^{2} \right)  \left(c^2 -{d}^{2} \right) }}  
 \endaligned \end{equation}

 For all $p\in \{(x,y,z,t): xyzt \ne
 0\}\cap Q_1^{-1}(0)$ and $Q_1=Q_1^{-1}(0)$ 
 positively oriented
 by the 
  inward
 normal,   directed by $-\nabla Q_1$,
  the principal curvatures satisfy
 $k_1(p) <0<  k_2(p) \leq k_3(p)$ and in the chart $(u,v,w)$ are given by:
 
 \begin{equation}\label{eq:AQ1p} k_1(u,v,w)={-\frac {abcd}{u\sqrt {uvw}}},\;\;
  k_2(u,v,w)={-\frac {abcd}{w\sqrt {uvw}}},\;\;
  k_3(u,v,w)={-\frac {abcd}{v\sqrt {uvw}}}
 \end{equation}
 
  The partially umbilic set ${\mathcal P}_{12}=  \emptyset$  and ${\mathcal P}_{23}=\{p: k_2(p)=k_3(p)\}$ is the union of   four open  arcs contained in the hyperplane $y=0$  
  The leaves of the principal foliation ${\mathcal F}_1$ are all 
  unbounded open arcs.
  
 \end{proposition}
 
 \begin{proof}
 It will be shown that  the quadric $
 Q_1
 $ belongs to a quadruply orthogonal family of quadrics.
 For $p = (x,y,z,t) \in \mathbb R^4$ and  $ \lambda \in  \mathbb R$,  let
 $$ Q(p,\lambda) =  \dfrac{x^2}{(a^2  +\lambda)} +\dfrac{ y^2}{(b^2 +\lambda)} -\dfrac{
 z^2}{(c^2  -\lambda)}-  \dfrac{t^2}{(d^2  -\lambda)}.  $$
 For   $p = (x,y,z,t) \in 
 Q_1\cap \{(x,y,z,t): xyzt\ne 0\} $  let $ u, v $  and
  $ w $ be the solutions of  system
  $$Q_1(p,u)=Q_1(v,p)=Q_1(p,w)=Q_1(p,0)=0.$$
  
 Solving   
 this linear system 
 with respect to the variables $x^2, y^2, z^2, t^2$ 
 leads to the parametrizations  of  
 equations \eqref{eq:AQ1R4} denoted by 
  $\psi(u,v,w)=(x,y,z,t) $,
    defined     in  the  connected  components of 
  $
  Q_1 \cap \{(x,y,z,t): xyzt\ne 0\}$.
   By symmetry 
   considerations,  it is sufficient to take 
  only the positive    orthant 
    $\{(x,y,z,t): x> 0, y>0, z>0, t>0\}$.
  
  Consider the parametrization $ \psi(u,v,w)=\varphi^{-1}(u,v,w)=(x,y,z,t)$,
  with $(x,y,z,t)$ in the positive orthant.

  Evaluating $g_{11} = (x_u)^2 + (y_u)^2 + (z_u)^2 + (t_u)^2 $,
  $g_{22} = (x_v)^2 + (y_v)^2 + (z_v)^2 + (t_v)^2 $,
  $g_{33} = (x_w)^2 + (y_w)^2 + (z_w)^2 + (t_w)^2 $,
  $g_{ij}= 0,~i\not= j $, 
   it follows that the first fundamental form is given by:
 
 \begin{equation}
 \label{eq:q1gij} 
\aligned I=&
 \frac 14\,{\frac {u \left( u-w \right)  \left( u-v \right) }{- \xi_1(u)}}
  du^2 +  \frac 14\,{\frac {v \left( v-w \right)  \left( u-v \right) }{ \xi_1(v)}} dv^2 \\
  +& \frac 14\,{\frac {w \left( v-w \right)  \left( u-w \right) }{ \xi_1(w)}} dw^2 
  \\
      \xi _1(\lambda ) = &( a^2  +
      \lambda ) ( b^2 + \lambda )( c^2 + \lambda )( d^2  - \lambda
      ).\endaligned \end{equation}

   Additional calculation shows that the second fundamental form with respect to $N$, $b_{11}=\langle (x_{uu}, y_{uu}, z_{uu}, t_{uu}), N\rangle,\ldots, b_{33}=\langle (x_{ww}, y_{ww},z_{ww},t_{ww}), N\rangle,$  is given by: 
   {\small 
   \begin{equation}
   \label{eq:q1bij} \aligned II=&  \frac 14 \frac{abcd}{\sqrt {uvw} }\left[ \,  \frac{ \left( u-w \right)  \left( u-v \right)  }{\xi_1(u) }
   du^2 -   \frac{ \left( u-v \right)  \left( v-w \right)  }{
      \xi_1(v) }   dv^2 -   \frac { \left( u-w \right)  \left( v-w \right)  }{ 
   \xi_1(w) }  dw^2 \right].\endaligned \end{equation}
 }

 Therefore the  coordinate  lines are principal curvature lines and
   the principal curvatures
  $  b_{ ii} /g_{ii}, \; (i =1,2,3)$ are expressed by:
 $$l = -\frac 1u \left(\dfrac{abcd}{\sqrt{ uvw } }\right),\;\;
 m = -\frac 1v \left(\dfrac{abcd}{\sqrt{ uvw } }\right),\;\;
 n = - \frac 1w \left(\dfrac{abcd}{\sqrt{uvw } }\right).$$
 Since $u >0 $ and $0>  v\geq  w $ it follows that  $n = m $ if, and only
 if, $ w = v = -b^2 $. Also   $l<0<n\leq m $.   Therefore, for  $p\in 
 Q_1 \cap \{(x,y,z,t): xyzt \ne 0\}$
 it follows that the principal curvatures satisfy
 $l=k_1(p)<0< k_2(p)=n\leq k_3(p)=m$.
 \end{proof}

 \begin{lemma}
 \label{lem:AQ1R4}  The  four open 
 arcs of partially umbilic points 
 located
 in the hyperplane  $y=0$ given in Proposition \ref{prop:AQ1R4} are contained in the two dimensional hyperboloid of one  sheet  
 $\frac{x^2}{a^2} +\frac{z^2}{c^2}- \frac{t^2}{d^2} =1 $.

 In the  parametrization 
 $$\alpha(u,v,w)=(au,bv,cw,  d\sqrt{u^2+v^2+w^2-1})$$
the partially umbilic set is given by:
 
 $${\frac { \left( {a}^{2}+{d}^{2} \right) }{{a}^{2}-{b}^{2}}} {u}^{2} -{
 \frac {\left( {c}^{2}+{d}^{2} \right) }{{b}^{2}-{c}^{2}}} {w}^{2} -1=0,\; u^2+w^2\geq 1, \;v=0.$$

 All  partially umbilic  arcs  are biregular  and  have vanishing torsion   only at $s=0$.

 \end{lemma}
 
 \begin{proof} From equations (\ref{eq:AQ1R4}) and (\ref{eq:AQ1p})  it follows that $k_2(p)=k_3(p)$ is defined by $v=w=-b^2$ and $u\in (d^2,\infty)$.
 From the symmetries of $
 Q_1$ it follows that one connected component is contained in the region $\{ (x,y,z,t):\; x>0, y=0, z>0\}.$
 Direct analysis shows that each connected  component is a biregular curve.
 \end{proof}

      \begin{lemma}\label{lem:CQ1R4} Let $\lambda \in (d^2,\infty)$ and consider the quartic surface
    $  Q_\lambda= Q_\lambda^{-1}(0)\cap Q_1$.
    
     Then $   Q_\lambda$ has two connected components, both are diffeomorphic to an ellipsoid with three different axes  and there exists a  principal parametrization of $  Q_\lambda$  such that
    the principal lines are the coordinates curves. Therefore, each connected component $ 
    Q_{\lambda}^i, \; (i=1,2)$,  of  
    $
    Q_\lambda$ is principally equivalent to  an ellipsoid of Monge (three different axes), i.e.,  there exists a homeomorphism $h_i:  Q_{\lambda}^i\to 
    q_0
    $  of equation (\ref{eq:q0}), preserving both principal foliations and singularities.

   \end{lemma}
   
   \begin{proof} 
   The intersection $Q_\lambda\cap  Q_1$ is   the same as $C_\lambda\cap   Q_1$, where
    \begin{equation}\aligned
     \label{eq:cone3}
     C_\lambda=& \left( \frac{1}{  {a}^{2}+\lambda}   -\frac{1}{a^{2} }\right) {x}^{2}+
       \left( \frac{1}{  {b}^{2}+\lambda}   -\frac{1}{b^{2} }\right) {y}^{2}+
   \left( \frac{1}{  {c}^{2}+\lambda}   -\frac{1}{c^{2} }\right){z}^{2}\\
   +&
       \left( \frac{1}{d^2}- \frac{1}{ {d}^{2}-\lambda} \right) {t}^{2}\endaligned
      \end{equation}

 For $\lambda \in (d^2,\infty)$  the quadratic form $Q_\lambda(x,y,z,t)$ has signature 0 $(++++)$, so it follows that $Q_\lambda$ has two connected components, each one  diffeomorphic to an ellipsoid and 
 which are contained in the regions $t>0$ and $t<0$.
 
  It follows from Propositions  \ref{prop:qelip} and  \ref{prop:AQ1R4}  that  $\psi_\lambda(v,w)=\psi(\lambda,v,w)$, $\psi:(-b^2,-c^2)\times (-a^2,-b^2)\to Q_\lambda$ is a parametrization  by principal curvature lines  of $Q_\lambda$
   in the region $Q_\lambda
   \cap \{(x,y,z,t),   t>0)\}$.  By symmetry, the component in the region $t<0$ can be parametrized 
 in similar way.
   
  The principal equivalence stated follows from the fact that  $Q_\lambda$  has  conformal principal charts as in    Lemma \ref{lem:q0r3conforme}. See  
  figure \ref{fig:cq}  and  
   also  \cite{bsbm14}.
   \end{proof}

     \begin{lemma}\label{lem:BQ1R4}
      Let $\lambda \in (-b^2,-c^2)$ and consider   the intersection of the quadric
      $$Q_\lambda(x,y,z,t)=\frac{x^2}{a^2+\lambda}+\frac{y^2}{b^2+\lambda}+
      \frac{z^2}{c^2+\lambda}-\frac{t^2}{d^2-\lambda}=1,\;\; a>  b>c>0,\; d>0,$$
  with the quadric $ Q_1=Q_0^{-1}(0)$.
  Let $ Q_\lambda= Q_\lambda^{-1}(0)\cap   Q_1$.
      Then $  Q_\lambda$ has two connected components, contained in the regions $z>0$ and $z<0$, both are diffeomorphic to a hyperboloid of one sheet with three different axes  and there exists a  principal parametrization of $  Q_\lambda$  such that
     the principal lines are the coordinates curves. Therefore, each connected component of  $ Q_\lambda$ is principally equivalent to a hyperboloid of one sheet.

    \end{lemma}

  \begin{proof} The principal chart defined by equation \eqref{eq:AQ1R4} has the first and second fundamental forms given by equations \eqref{eq:q1gij} and \eqref{eq:q1bij}.
  
  The intersection $Q_\lambda\cap Q_1$ is also the same as $C_\lambda\cap   Q_1$, where
   
  \begin{equation}\aligned
  \label{eq:cone4}
  C_\lambda=& \left( \frac{1}{  {a}^{2}+\lambda}   -\frac{1}{a^{2} }\right) {x}^{2}+
    \left( \frac{1}{  {b}^{2}+\lambda}   -\frac{1}{b^{2} }\right) {y}^{2}+
\left( \frac{1}{  {c}^{2}+\lambda}   -\frac{1}{c^{2} }\right){z}^{2}\\
+&
    \left( \frac{1}{d^2}- \frac{1}{ {d}^{2}-\lambda} \right) {t}^{2}\endaligned
   \end{equation}
   
   The quadratic form $C_\lambda(x,y,z,t)$, $\lambda\in (-b^2,-c^2)$  has signature 1 $(++ - +)$ in the    
 sense of Morse.
  %ron: conferido o index 1
  
   Therefore, for   $\lambda \in (-b^2,-c^2)$, 
    it follows from Propositions  \ref{prop:q1folha} and \ref{prop:AQ1R4}  that $\psi_\lambda(u,w)=\psi(u,  \lambda,w)$, $\psi:(d^2,\infty)\times (-a^2,-b^2)\to Q_\lambda$ is a parametrization of $Q_\lambda$
  in the region $Q_\lambda
  \cap \{(x,y,z,t),   z>0)\}$ by principal curvature lines. The other component, contained  in the region $z<0$, can be parametrized similarly.
  It follows that $Q_\lambda$
    has two connected components, one contained in the region $z>0$ and the other in the region $z<0$.
   The conclusion of the proof is similar to that of Lemma \ref{lem:CQ1R4}.

     The principal equivalence stated follows from the fact that  $Q_\lambda$  has  conformal principal charts similar 
     to those established  in  Lemma \ref{lem:q0r3conforme}. See figure \ref{fig:cq} and also  \cite{bsbm14}.
  \end{proof}

     \begin{lemma}\label{lem:DQ1R4}
      Let $\lambda \in (-a^2,-b^2)$ and consider   the intersection of the quadric
      $$Q_\lambda(x,y,z,t)=\frac{x^2}{a^2+\lambda}+\frac{y^2}{b^2+\lambda}+
      \frac{z^2}{c^2+\lambda}-\frac{t^2}{d^2-\lambda}=1,\;\; a > b>c>0,\; d>0,$$
  with the quadric $ Q_1=Q_0^{-1}(0)$.
  
   Let $  Q_\lambda= Q_\lambda^{-1}(0)\cap Q_1$.
        Then $ Q_\lambda$ has two connected components, contained in the regions $x>0$ and $x<0$, both are diffeomorphic to a hyperboloid of one sheet  with three different axes  and there exists a  principal parametrization of $  Q_\lambda$  such that
       the principal lines are the coordinates curves. Therefore, each connected component of  $  Q_\lambda$ is principally equivalent to a hyperboloid of one  sheet.
  
   \end{lemma}
   
   \begin{proof}
    The intersection $Q_\lambda\cap   Q_1$ is   the same as $C_\lambda\cap   Q_1$, where
      \begin{equation}\aligned
       \label{eq:cone5}
       C_\lambda= &\left( \frac{1}{  {a}^{2}+\lambda}   -\frac{1}{a^{2} }\right) {x}^{2}+
         \left( \frac{1}{  {b}^{2}+\lambda}   -\frac{1}{b^{2} }\right) {y}^{2}+
     \left( \frac{1}{  {c}^{2}+\lambda}   -\frac{1}{c^{2} }\right){z}^{2}\\
     +&
         \left( \frac{1}{d^2}- \frac{1}{ {d}^{2}-\lambda} \right) {t}^{2}\endaligned
        \end{equation}
        
         Therefore, for   $\lambda \in (-a^2,-b ^2)$,   it follows from Propositions  \ref{prop:q1folha} and \ref{prop:AQ1R4}  that $\psi_\lambda(u,v)=\psi( u,v, \lambda)$, $\psi:(d^2,\infty)\times (-b^2,-c^2)\to Q_\lambda$ is a parametrization of $Q_\lambda$
          in the region $Q_\lambda
          \cap \{(x,y,z,t),   x>0)\}$ by principal curvature lines. 
    The principal equivalence stated follows from the fact that  $  Q_\lambda$  has  conformal principal charts 
    as obtained in    Lemma \ref{lem:q0r3conforme}. See figure \ref{fig:cq} and  also  \cite{bsbm14}.
   \end{proof}

   The results above of this section are summarized in the following theorem.
   
    \begin{theorem} \label{th:Q1R4}
     The umbilic set  of  $ Q_1$ 
     in equation (\ref{eq:Q1})
     is empty and its partially umbilic set  consists of four curves
     ${\mathcal P}_{23}^1, \; {\mathcal P}_{23}^2,\; {\mathcal P}_{23}^3$ and $ {\mathcal P}_{23}^4$.
    
     Consider the parametrization $$\alpha(u,v,w)=(au,bv,cw,  d\sqrt{u^2+v^2+w^2-1}).$$
      
      In this parametrization the partially umbilic set is  given by:
      
      $${\frac { \left( {a}^{2}+{d}^{2} \right) }{{a}^{2}-{b}^{2}}} {u}^{2} - 
      \frac {\left( {c}^{2}+{d}^{2} \right) }{ b^{2}-{c}^{2} }  w^2 -1=0,\; u^2+w^2\geq 1, \;v=0.$$

       \noindent i)  \; All leaves of the  principal foliation  ${\mathcal F}_1  $
   are open arcs and diffeomorphic to  $\mathbb R.\;$

     The partially umbilic curves ${\mathcal P}_{23}^i$ , (i=1,\ldots , 4)
     are of type  $D_1$, its  partially umbilic 
     separatrix surfaces 
       span  open  bands  $W_1,\; W_2,\; W_3,\; W_4 $
     such that $\partial W_1=   {\mathcal P}_{23}^1\cup  {\mathcal P}_{23}^2$,    $\partial W_2=  {\mathcal P}_{23}^2\cup  {\mathcal P}_{23}^3$, 
     $  \partial W_3= {\mathcal P}_{23}^3\cup  {\mathcal P}_{23}^4$  and
     $\partial W_4=  {\mathcal P}_{23}^1\cup  {\mathcal P}_{23}^4$.
     
     \noindent   
     \noindent ii)   All the leaves of    the principal foliations ${\mathcal F}_2$ and ${\mathcal F}_3$ outside the bands $W_1,\; W_2,\; W_3$ and $ W_4 $ are compact.
      See illustration in Fig. \ref{fig:q1}.

   \begin{figure}[h]
   \begin{center}
     \def\svgwidth{0.90\textwidth}
          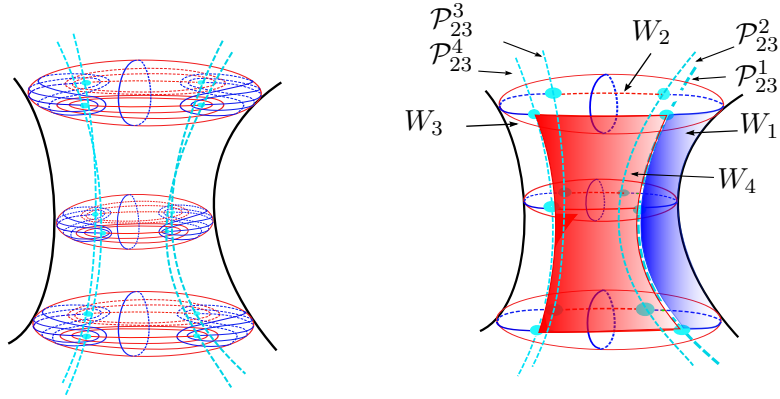
       \caption{Principal foliations $\mathcal{F}_i $ on the Ellipsoidal Hyperboloid of one Sheet. 
       }
     \label{fig:q1}
       \end{center}
   \end{figure}
   \end{theorem}

\section{Toroidal Hyperboloid in $\mathbb R^4$
}\label{ss:q2abcd}
 
\noindent In this section will be established  the global behavior of principal lines in the quadric $ Q_2$  of signature 2 $(++ - - ) $.

\begin{proposition}\label{prop:AQ2R4}
The quadric  $ Q_2 $  given by
\begin{equation} \label{eq:Q2}
Q_2(x,y,z,t)=\dfrac{x^2}{a^2 } +\dfrac{ y^2}{b^2 } -\dfrac{
z^2}{c^2 } - \dfrac{t^2}{d^2 } - 1=0, \;\;\; a>b>0, \; c>d>0,
\end{equation}
 is diffeomorphic to 
 ${\mathbb D}^2\times \mathbb S^1$
 and 
  has
sixteen  principal charts $(u,v,w)=\varphi_i(x,y,z,t)$, $u>v>w$,  where
$$\varphi_i^{-1}:   (c^2,\infty)  \times (d^2,c^2) \times (-a^2,-b^2) \to 
\{(x,y,z,t): xyzt\ne 0\}\cap Q_ 2^{-1}(0) $$
 is defined by equation
\eqref{eq:AQ2R4}.

\begin{equation}\label{eq:AQ2R4} \aligned x^2 &= {\frac { \left( {a}^{2}+w \right)  \left( {a}^{2}+v \right)  \left( {a
}^{2}+u \right) {a}^{2}}{ \left( {d}^{2}+{a}^{2} \right)  \left( {c}^
{2}+{a}^{2} \right)  \left( a^2-{b}^{2} \right) }},
\; y^2 = -{\frac {{b}^{2} \left( {b}^{2}+w \right)  \left( {b}^{2}+v \right), 
 \left( {b}^{2}+u \right) }{ \left( {d}^{2}+{b}^{2} \right)  \left( {c
}^{2}+{b}^{2} \right)  \left( -{b}^{2}+{a}^{2} \right) }},
\\
z^2 &=-{\frac {{c}^{2} \left( {c}^{2}-w \right)  \left( {c}^{2}-v \right) 
 \left( {c}^{2}-u \right) }{ \left( {c}^{2}+{b}^{2} \right)  \left( {c
}^{2}+{a}^{2} \right)  \left(c^2 -{d}^{2}\right) }}
,\;
 t^2 ={\frac {{d}^{2} \left( {d}^{2}-w \right)  \left( {d}^{2}-v \right) 
   \left( {d}^{2}-u \right) }{ \left( {d}^{2}+{a}^{2} \right)  \left( {d
  }^{2}+{b}^{2} \right)  \left(c^2 -{d}^{2} \right) }}.  
\endaligned \end{equation}

For all $p\in \{(x,y,z,t): xyzt \ne
0\}\cap Q_2^{-1}(0)$ and $Q_2=Q_2^{-1}(0)$ 
positively oriented
 by the normal 
 $N=\nabla Q_2 $ 
 the principal curvatures satisfy
$k_1(p)<0<k_2(p)\leq k_3(p)$ and in the chart $(u,v,w)$ are given by:

\begin{equation}\label{eq:AQ2p}
k_1(u,v,w)={\frac {abcd}{w\sqrt {-uvw}}},\; k_2(u,v,w)={\frac {abcd}{u\sqrt {-uvw}}},\;
k_3(u,v,w)={\frac {abcd}{v\sqrt {-uvw}}},\;
\end{equation}

 The partially umbilic set ${\mathcal P}_{23}=\{p: k_2(p)=k_3(p)\}$ is the union of  two closed curves contained in the hyperplane $z=0$ and ${\mathcal P}_{12}=\{p: k_1(p)=k_2(p)\}=\emptyset$.
 
 The leaves of the principal foliation ${\mathcal F}_1$ are all closed.
 
\end{proposition}

\begin{proof} Let 
$$ Q_2(p,\lambda) =  \dfrac{x^2}{(a^2  +\lambda)} +\dfrac{ y^2}{(b^2 +\lambda)} -\dfrac{
 z^2}{(c^2  -\lambda)}-  \dfrac{t^2}{(d^2  -\lambda)} $$
 and consider the system $Q_2(p,u)=Q_2(v,p)=Q_2(p,w)=Q_2(p,0)=0.$
 
As in the proof of  Proposition \ref{prop:AQ1R4} it follows that the solution of the linear system above in the variables $x^2$, $y^2$, $z^2$ and $w^2$ is given by equation \eqref{eq:AQ2R4}.
 
 So 
the map $\varphi:  Q_2\cap \{(x,y,z,t): xyzt\ne 0\} \to (c^2,\infty)\times (d^2,c^2)\times  (-a^2,-b^2)   $ 
is well defined.

The map 
$$\varphi:   Q_2 \cap \{(x,y,z,t): xyzt\ne 0\}\to  (c^2,\infty)\times (d^2,c^2)\times  (-a^2,-b^2), \; $$
$\varphi (x,y,z,t) = ( u,v,w) $ is a regular covering
which defines a chart in each orthant of the quadric  $ Q_2$.

So, equations in 
 \eqref{eq:AQ2R4}   define  parametrizations $\psi(u,v,w)=(x,y,z,t) $ of the  connected  
 components of the region $ Q_2 \cap \{(x,y,z,t): xyzt\ne 0\}$.
 By symmetry, it is sufficient to consider only the positive orthant $\{(x,y,z,t): x> 0, y>0, z>0, t>0\}$.

Consider the parametrization $ \psi(u,v,w)=\varphi^{-1}(u,v,w)=(x,y,z,t)$,
with $(x,y,z,t)$ in the positive orthant.

Evaluating $g_{11}  = (x_u)^2 + (y_u)^2 + (z_u)^2 + (w_u)^2 $,
$g_{22}  = (x_v)^2 + (y_v)^2 + (z_v)^2 + (w_v)^2 $ and
$g_{33} = (x_w)^2 + (y_w)^2 + (z_w)^2 + (t_w)^2 $  and observing that
$g_{ij} = 0,~i\not= j $,  it follows that the first fundamental form is given by:

\begin{equation}\label{eq:gijQ2}
\aligned  I =& \frac 14\,{\frac {u \left( u-w \right)  \left( u-v \right) }{ \xi_2(u) }}
 du^2  
-\frac 14 \,{\frac {v \left( v-w \right)  \left( u-v \right) }{ \xi_2(v) }}
dv^2 \\
+&  \frac 14\,{\frac {w \left( v-w \right)  \left( u-w \right) }{ \xi_2(w) }}
dw^2\\
 \xi_2 (\lambda ) =& ( a^2  +
\lambda ) ( b^2 + \lambda )( c^2 - \lambda )( d^2  - \lambda
)\endaligned  .\end{equation}

 Let    $$N=\frac{\nabla Q_2}{|\nabla Q_2|}=\frac{abcd}{\sqrt{-uvw}} \left(\frac{x(u,v,w)}{a^2}, \frac{y(u,v,w)}{b^2},-\frac{z(u,v,w)}{c^2}, -\frac{t(u,v,w)}{d^2}\right).$$

Similar and straightforward calculation shows that the second fundamental form with respect to $N$ is given by:
{\small
\begin{equation}\label{eq:bijQ2} II = \frac{abcd}{4\sqrt{-uvw}}  \left[     \,{\frac {  \left( u-w \right)  \left( u-v \right) }{ \xi_2(u) }}
 du^2  +
  \,{\frac {  \left( v-w \right)  \left( u-v \right) }{ \xi_2(v) }}
dv^2 +   \,{\frac {  \left( v-w \right)  \left( u-w \right) }{ \xi_2(w) }}
dw^2  \right] \end{equation}
}
Therefore the  coordinate  lines are principal curvature lines and
  the principal curvatures
 $  b_{ ii} /g_{ii}, \; (i =1,2,3)$ are given by:
$$l = \frac 1u \left(\dfrac{abcd}{\sqrt{- uvw } }\right),\quad
m = \frac 1v \left(\dfrac{abcd}{\sqrt{ -uvw } }\right),\quad
n = \frac 1w \left(\dfrac{abcd}{\sqrt{- uvw } }\right).$$
Since $u \geq v>0  > w $ it follows that  $l = m $ if, and only
if, $ u = v = c^2 $. Also   $n<0<l\leq m $.   Therefore, for  $p\in Q_2  \cap \{(x,y,z,t): xyzt\ne 0\}$
it follows that the principal curvatures satisfy
$n=k_1(p)<k_2(p)=l\leq k_3(p)=m$.

The parametrization

$$\beta(u,v,w)=\left( a\cos u\cosh w ,b\sin u \cosh w ,c\cos v\sinh w ,d\sin v \sinh w\right)$$
 with $w\geq 0, u,v\in [0,2\pi],$
shows that the quadric $ Q_2$ is diffeomorphic to $\mathbb D^2\times \mathbb S^1$.

 \end{proof}

\begin{lemma}
\label{lem:pumbq2}  The closed  curves of partially umbilic points contained in the hyperplane  $z=0$ given in Proposition \ref{prop:AQ2R4} are contained in the two dimensional hyperboloid of one
sheet
$\frac{x^2}{a^2} +\frac{y^2}{b^2}- \frac{t^2}{d^2} =1 $ and $z=0$, in the regions $t>0$ and $t<0$ and are principal lines of this quadric as a surface  of $\mathbb R^3$.

 Both curves are biregular and 
 have  torsion zero when  intersecting  the coordinate planes of $(x,y,0,t).$
 
Consider the parametrizations of $ Q_2$ below.
\begin{equation}
\alpha_\pm(u,v,w)=\left(au,bv,cw,\pm d \sqrt {{u}^{2}+{v}^{2}-{w}^{2}-1} \right).
\end{equation}
In this parametrization the partially umbilic set is given by:

$$E(u,v)={\frac { \left( {d}^{2}+{a}^{2} \right) {u}^{2}}{{c}^{2}+{a}^{2}}}+{
\frac { \left( {b}^{2}+{d}^{2} \right) {v}^{2}}{{b}^{2}+{c}^{2}}}-1=0,\;\; w=0.$$

\end{lemma}

\begin{proof}
From equations \eqref{eq:AQ2R4} and \eqref{eq:AQ2p}  it follows that the partially umbilic curves are defined by 
$u=v=c^2$, $w\in (-a^2,-b^2) $ and so it follows that:

$$\aligned x^2=& {\frac { \left( {a}^{2}+w \right)  \left( {c}^{2}+{a}^{2} \right) {a}
^{2}}{ \left( {d}^{2}+{a}^{2} \right)  \left( {a}^{2}-{b}^{2} \right) 
}}, \; y^2= -{\frac {{b}^{2} \left( {b}^{2}+w \right)  \left( {c}^{2}+{b}^{2}
 \right) }{ \left( {d}^{2}+{b}^{2} \right)  \left( {a}^{2}-{b}^{2}
 \right) }},\;
  z=0,\\ t^2=&{\frac { \left( c^2-d^2 \right)    \left( {d
 }^{2}-w \right) {d}^{2}}{ \left( {d}^{2}+{a}^{2} \right)  \left( {d}^{
 2}+{b}^{2} \right) }}, \; \; -a^2<w<-b^2. \endaligned $$

Consider now the parametrization 
$$\alpha (u,v,w)=\left(au,bv,cw,  d\sqrt {{u}^{2}+{v}^{2}-{w}^{2}-1}\right)$$
defined in $\{(u,v,w)\in \mathbb R^3: \;\Delta= u^2+v^2-w^2-1> 0\}$.
The  positive normal vector is proportional to:
 
$$N=\left( -bcd u,-acdv,abdw,abc\sqrt {{u}^{2}+{v}^{2}-{w}^{2}-1}\right).$$

Direct analysis shows that the first and second fundamental forms are given by:

\begin{equation}\aligned 
g_{11} =&a^2+\frac{d^2 u^2}{\Delta},\;\hskip .8cm  g_{22} = b^2+\frac{d^2 v^2}{\Delta},\; \hskip .8cm  g_{33}= c^2+\frac{d^2 w^2}{\Delta},\\
g_{ 12}=&\frac{ d^2 uv}{\Delta},\;\hskip 1.5cm   g_{13}=-\frac{ d^2 uw}{\Delta},\; \hskip 1.5cm  g_{23}=-\frac{ d^2 vw}{\Delta}\\
b_{11}=&\frac{abcd(v^2-w^2-1)}{\Delta^{\frac 32}},\; b_{12}= -\frac{abcd uv}{\Delta^{\frac 32}}, \hskip .7cm b_{13}=  \frac{abcd uw}{\Delta^{\frac 32}}\\
b_{22}=&\frac{abcd(u^2-w^2-1)}{\Delta^{\frac 32}},\; b_{23}=  \frac{abcd vw}{\Delta^{\frac 32}},\hskip .4cm b_{33}=  -\frac{abcd(u^2+v^2-1)}{\Delta^{\frac 32}}.
\endaligned
\end{equation}

From proposition \ref{prop:AQ2R4} the partially umbilic set is contained in the hyperplane $z=0$, so $w=0$. Direct calculation shows that the partially umbilic set is defined by equations 
$$(b_{11}g_{22}-b_{22}g_{11})(u,v,0)=0, \; (g_{12}b_{22}- g_{22}b_{12})(u,v,0)=0. $$

\end{proof}

   \begin{lemma}\label{lem:CQ2R4} Let $\lambda \in (c^2,\infty)$ and consider the quartic surface
 $ Q_\lambda= Q_\lambda^{-1}(0)\cap Q_2$.
 
  Then $\  Q_\lambda$ is diffeomorphic to a bidimensional torus of revolution and there exists a conformal principal parametrization of $ Q_\lambda$  such that
 the principal lines are the coordinates curves. Therefore,   $ Q_\lambda$ is principally equivalent to a torus of revolution of $\mathbb R^3$.

\end{lemma}

\begin{proof}  
For $\lambda \in (c^2,\infty)$  the quadratic form $Q_\lambda(x,y,z,t)$ has signature 0 $(++++)$, so it follows that $Q_\lambda$ has only one connected component diffeomorphic to a torus.

    The principal equivalence stated follows from the fact that  $Q_\lambda$  has  conformal principal charts that can be obtained as in the proof of Lemma \ref{lem:q0r3conforme}. See also  \cite{bsbm14}.
\end{proof}

   \begin{lemma}\label{lem:BQ2R4}
    Let $\lambda \in (d^2,c^2)$ and consider   the intersection of the quadric
    $$Q_\lambda(x,y,z,t)=\frac{x^2}{a^2+\lambda}+\frac{y^2}{b^2+\lambda}-
    \frac{z^2}{c^2-\lambda}-\frac{t^2}{d^2-\lambda}=1,\;\; a>b>c>d>0,$$
with the quadric $  Q_2=Q_0^{-1}(0)$.
Let $  Q_\lambda= Q_\lambda^{-1}(0)\cap  Q_2$.
        Then $  Q_\lambda$ has two connected components, both are diffeomorphic to a hyperboloid of one sheet with three different axes  and there exists a  principal parametrization of $ Q_\lambda$  such that
       the principal lines are the coordinates curves. Therefore, each connected component of  $  Q_\lambda$ is principally equivalent to a hyperboloid of one sheet.

  \end{lemma}

\begin{proof} Similar to the proof  of lemma \ref{lem:BQ1R4}
\end{proof}

   \begin{lemma}\label{lem:DQ2R4}
    Let $\lambda \in (-a^2,-b^2)$ and consider   the intersection of the quadric
    $$Q_\lambda(x,y,z,t)=\frac{x^2}{a^2+\lambda}+\frac{y^2}{b^2+\lambda}-
    \frac{z^2}{c^2-\lambda}-\frac{w^2}{d^2-\lambda}=1,\;\; a>b>c>d>0,$$
with the quadric $  Q_2=Q_0^{-1}(0)$.
 Let $ Q_\lambda= Q_\lambda^{-1}(0)\cap  Q_2$.
        Then $  Q_\lambda$ has four connected components and each one diffeomorphic to a leaf of a hyperboloid of two  sheets with three different axes  and there exists a  principal parametrization of $  Q_\lambda$  such that
       the principal lines are the coordinates curves. Therefore, each connected component of  $  Q_\lambda$ is principally equivalent to a leaf of a  hyperboloid of two sheets.
 \end{lemma}
 \begin{proof} Similar to the proof of Lemma \ref{lem:DQ1R4}.
 From propositions \ref{prop:q2folhas} and \ref{eq:AQ2R4} it follows  
  that $\psi_\lambda(u,v)=\psi( u,v, \lambda)$, $\psi:(c^2,\infty)\times (d^2,c^2)\to Q_\lambda$ is a parametrization of $Q_\lambda$
           in the region $Q_\lambda
           \cap \{(x,y,z,t),   x>0, y>0)\}$ by principal curvature lines. 
 
 \end{proof}

     The results above are summarized in the following theorem.
     
     \begin{theorem} \label{th:Q2R4}
      The umbilic set  of    $Q_2$ in equation ( \ref{eq:Q2}) is empty and its partially umbilic set  consists of two curves
            ${\mathcal P}_{23}^1, \; {\mathcal P}_{23}^2$.
            
      Consider the parametrization  of $\ Q_2$ below.
      \begin{equation}
      \alpha_\pm(u,v,w)=\left(au,bv,cw, d \sqrt {{u}^{2}+{v}^{2}-{w}^{2}-1} \right).
      \end{equation}
      In this parametrization the partially umbilic set is given by:
      
      $$E(u,v)={\frac { \left( {d}^{2}+{a}^{2} \right) {u}^{2}}{{c}^{2}+{a}^{2}}}+{
      \frac { \left( {b}^{2}+{d}^{2} \right) {v}^{2}}{{b}^{2}+{c}^{2}}}-1=0,\;\; w=0.$$

      The partially umbilic curves ${\mathcal P}_{23}^1$  and  ${\mathcal P}_{23}^2$  
      are of type  $D_1$, its  partially umbilic 
      separatrix surfaces 
        span  open  bands  ${ W_{23}}^1 $
      such that $\partial{ W_{23}}^1=  {\mathcal P}_{23}^1\cup  {\mathcal P}_{23}^2$.
      
        \noindent i)  \; All leaves of the  principal foliation  ${\mathcal F}_3 $, outside the umbilic separatrix surface, are closed.

      \noindent   
        \noindent ii)  \; All leaves of the  principal foliation  ${\mathcal F}_2 $
          are open arcs and diffeomorphic to  $\mathbb R.\;$
      
      \noindent iii)   All the leaves of    the principal foliations ${\mathcal F}_1$   are compact.
       See illustration in Fig. \ref{fig:Q2global}.

    \end{theorem}
  \begin{figure}[ht]
  \begin{center}
   \def\svgwidth{0.60\textwidth}
    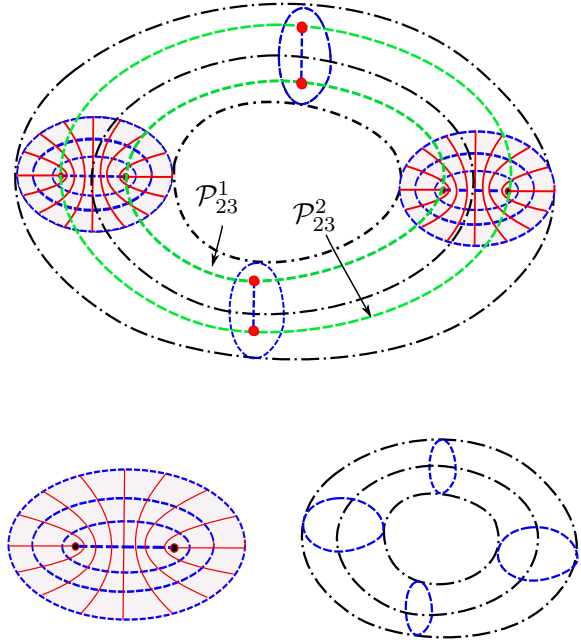
  \caption{Principal lines of   the quartic $Q_2$ of index 2. \label{fig:Q2global} }
  \end{center}
  \end{figure}

 \section{
 Ellipsoidal Hyperboloid with two sheets in $\mathbb R^4$
} \label{ss:q3abcd}
 
    \noindent In this section will be established  the global behavior of principal lines in the quadric 
    $ Q_3$
     of signature 3 $(+ -  - - ) $.

 \begin{proposition}\label{prop:AQ3R4}
 The quadric  
 given by
 \begin{equation} \label{eq:Q3}
  Q_3(x,y,z,t)=\dfrac{x^2}{a^2 } -\dfrac{ y^2}{b^2 } -\dfrac{
 z^2}{c^2 } - \dfrac{t^2}{d^2 } - 1=0, \;\;\; a>0, \; b>c>d>0,
 \end{equation} 
 has two connected components, each one is   diffeomorphic to $  \mathbb R^3$ and has
 sixteen  principal charts $(u,v,w)=\varphi_i(x,y,z,t)$, $u>v>w$,  where
 $$\varphi_i^{-1}:   (b^2,\infty)  \times (c^2,b^2) \times (d^2, c^2) \to 
 \{(x,y,z,t): xyzt\ne 0\}\cap Q_3^{-1}(0) $$
  is defined by equation
 \eqref{eq:AQ3R4}.

 \begin{equation}\label{eq:AQ3R4} \aligned x^2 &= {\frac { a^2\left( {a}^{2}+w \right)  \left( {a}^{2}+v \right)  \left( {a
 }^{2}+u \right) }{ \left( {d}^{2}+{a}^{2} \right)  \left( {c}^{
 2}+{a}^{2} \right)  \left( {b}^{2}+{a}^{2} \right) }}
 ,\;
  y^2 = -{\frac {{b}^{2} \left( {b}^{2}-w \right)  \left( {b}^{2}-v \right) 
  \left( {b}^{2}-u \right) }{ \left( {b}^{2}-{d}^{2} \right)  \left( {
 b}^{2}+{a}^{2} \right)  \left( b^2-{c}^{2} \right) }}
 \\
 z^2 &={\frac {{c}^{2} \left( {c}^{2}-w \right)  \left( {c}^{2}-v \right) 
  \left( {c}^{2}-u \right) }{ \left( -{d}^{2}+{c}^{2} \right)  \left( -
 {c}^{2}+{b}^{2} \right)  \left( {c}^{2}+{a}^{2} \right) }}
 ,\;
  t^2 =-{\frac {{d}^{2} \left( {d}^{2}-w \right)  \left( {d}^{2}-v \right) 
   \left( {d}^{2}-u \right) }{ \left( {d}^{2}+{a}^{2} \right)  \left( -{
  d}^{2}+{c}^{2} \right)  \left( -{d}^{2}+{b}^{2} \right) }} .
 \endaligned \end{equation}

 For all $p\in \{(x,y,z,t): xyzt \ne
 0\}\cap Q_ 3^{-1}(0)$ and $Q_3=Q_ 3^{-1}(0)$ positively oriented, the principal curvatures satisfy
 $0<k_1(p)\leq<k_2(p)\leq k_3(p)$ and in the chart $(u,v,w)$ are given by:
 
 \begin{equation}\label{eq:AQ3p}
 k_1(u,v,w)={\frac {abcd}{u\sqrt {uvw}}},\; k_2(u,v,w)={\frac {abcd}{u\sqrt {uvw}}},\;
 k_3(u,v,w)={\frac {abcd}{w\sqrt {uvw}}}.
 \end{equation}
 
  The partially umbilic set ${\mathcal P}_{12}=\{p: k_1(p)=k_2(p)\}$ is the union of two  closed curves contained in the hyperplane $y=0$, one contained in the region  $\{x>a\}$ and the other two are contained in the region$\{x<-a\}$  and ${\mathcal P}_{23}=\{p: k_2(p)=k_3(p)\}$ is the union of  four open 
  arcs 
  contained in the hyperplane $z=0$,  two are contained in the region $\{x>a\}$ and the other two are contained in the region$\{x<-a\}$.
  
 \end{proposition}
 \begin{proof}
Let 
 $$ Q_3(p,\lambda) =  \dfrac{x^2}{(a^2  -\lambda)} -\dfrac{ y^2}{(b^2 -\lambda)} -\dfrac{
  z^2}{(c^2  -\lambda)}-  \dfrac{t^2}{(d^2  -\lambda)}.  $$
  
  and consider the system $Q_3(p,u)=Q_3(v,p)=Q_3(p,w)=Q_3(p,0)=0.$
  
 As in the proof of  Proposition \ref{prop:AQ1R4} it follows that the solution of the linear system above in the variables $x^2$, $y^2$, $z^2$ and $w^2$ is given by equation \eqref{eq:AQ3R4}.
  
  So 
 the map $\varphi: Q_3\cap \{(x,y,z,t): xyzt\ne 0\} \to (b^2,\infty)  \times (c^2,b^2) \times (d^2, c^2)  $ 
 is well defined.

 The map 
 $$\varphi: Q_3 \cap \{(x,y,z,t): xyzt\ne 0\}\to  (b^2,\infty)  \times (c^2,b^2) \times (d^2, c^2) , \; $$
 $\varphi (x,y,z,t) = ( u,v,w) $ is a regular covering
 which defines a chart
 in each orthant of the quadric  $ Q_3$.

 So, equations in 
  \eqref{eq:AQ3R4}   define  parametrizations $\psi(u,v,w)=(x,y,z,t) $ of the  connected  components of the region $Q_3  \cap \{(x,y,z,t): xyzt\ne 0\}$.
  By symmetry, it is sufficient to consider only the positive
  orthant 
  $\{(x,y,z,t): x> 0,\; y>0,\; z>0,\; t>0\}$.
 
 Consider the parametrization $ \psi(u,v,w)=\varphi^{-1}(u,v,w)=(x,y,z,t)$,
 with $(x,y,z,t)$ in the positive orthant.  The fundamental forms of
 $ Q_3$ will be evaluated and expressed in terms of the function
 
 \begin{equation}\label{eq:xiq3} \xi_3 (\lambda ) = ( a^2  +
 \lambda ) ( b^2 - \lambda )( c^2 - \lambda )( d^2  - \lambda
 ).\;\end{equation}

 Evaluating $g_{11} = (x_u)^2 + (y_u)^2 + (z_u)^2 + (w_u)^2 $,
 $g_{22} = (x_v)^2 + (y_v)^2 + (z_v)^2 + (w_v)^2 $ and
 $g_{33}= (x_w)^2 + (y_w)^2 + (z_w)^2 + (t_w)^2 $  and observing that
 $g_{ij} = 0,~i\not= j$,  it follows that the first fundamental form is given by:

 \begin{equation}\label{eq:gijQ3} I = \frac 14\,{\frac {u \left( u-w \right)  \left( u-v \right) }{ \xi_3(u) }}
  du^2  
 -\frac 14 \,{\frac {v \left( v-w \right)  \left( u-v \right) }{ \xi_3(v) }}
 dv^2 +  \frac 14\,{\frac {w \left( v-w \right)  \left( u-w \right) }{ \xi_3(w) }}
 dw^2  .\end{equation}

  Let    $$N=\frac{\nabla Q_3}{|\nabla Q_3|}=\frac{abcd}{\sqrt{uvw}} \left(\frac{x(u,v,w)}{a^2},- \frac{y(u,v,w)}{b^2},-\frac{z(u,v,w)}{c^2}, -\frac{t(u,v,w)}{d^2}\right).$$
 
 Similar and straightforward calculation shows that the second fundamental form with respect to $N$ is given by:
 {\small
 \begin{equation}\label{eq:bijQ3} II = \frac{abcd}{4\sqrt{uvw}}  \left[     \,{\frac {  \left( u-w \right)  \left( u-v \right) }{ \xi_3(u) }}
  du^2  
   \,{\frac {  \left( v-w \right)  \left( u-v \right) }{ \xi_3(v) }}
 dv^2 +   \,{\frac {  \left( v-w \right)  \left( u-w \right) }{ \xi_3(w) }}
 dw^2  \right] .\end{equation}
 }
 Therefore the  coordinate  lines are principal curvature lines and
   the principal curvatures
  $  b_{ ii} /g_{ii}i, \; (i =1,2,3)$ are given by:
 $$l = \frac 1u \left(\dfrac{abcd}{\sqrt{uvw } }\right),\quad
 m = \frac 1v \left(\dfrac{abcd}{\sqrt{ uvw } }\right),\quad
 n = \frac 1w \left(\dfrac{abcd}{\sqrt{uvw } }\right).$$
 Since $u \geq v \geq w>0 $ it follows that  $l = m $ if, and only
 if, $ u = v = b^2 $. Also   $n=m $ if and only if $v=w=c^2$.   Therefore, for  $p\in Q_3 \cap \{(x,y,z,t): xyzt\ne 0\}$
 it follows that the principal curvatures satisfy
 $l=k_1(p)<k_2(p)=m\leq k_3(p)=n$.
  \end{proof}
  
  \begin{lemma}
  \label{lem:AQ3R4}  The    partially umbilic set, contained in the hyperplane  $y=0$ given in Proposition \ref{prop:AQ1R4} is  contained in the two dimensional hyperboloid 
  with two sheets
  $\frac{x^2}{a^2} -\frac{z^2}{c^2}- \frac{t^2}{d^2} =1 $ and the partially umbilic set, contained in the hyperplane  $z=0$ given in Proposition \ref{prop:AQ1R4},  is  contained in the two dimensional hyperboloid of two leaves $\frac{x^2}{a^2} -\frac{y^2}{b^2}- \frac{t^2}{d^2} =1$.

   In the parametrization 
   $$\alpha(u,v,w)=(a\sqrt{(1+u^2+v^2+w^2} ),bu,cv,dw)$$
   with positive normal vector proportional to:   
   {\small  
   $$N=\left( bcd,-{\frac {a cd u}{\sqrt {{u}^{2}+{v}^{2}+{w}^{2}+1}}},-{\frac {a bdv
   }{\sqrt {{u}^{2}+{v}^{2}+{w}^{2}+1}}},-{\frac {a bcw}{\sqrt {{u}^{2}+{v
   }^{2}+{w}^{2}+1}}}\right)$$}
   
   the partially umbilic set is defined by the  conics 
   
   \begin{equation}\label{eq:puq3}
   \aligned
   E(v,w)=& {\frac { \left( {a}^{2}+{c}^{2} \right) {v}^{2}}{{b}^{2}-{c}^{2}}}+{
   \frac { \left( {a}^{2}+{d}^{2} \right) {w}^{2}}{{b}^{2}-{d}^{2}}}-1,\;\; u=0\\
  H(u,w)=&-{\frac { \left( {a}^{2}+{b}^{2} \right) {u}^{2}}{{b}^{2}-{c}^{2}}}+{
  \frac { \left( {a}^{2}+{d}^{2} \right) {w}^{2}}{{c}^{2}-{d}^{2}}}-1,
  \;\; v=0\endaligned 
   \end{equation}
 
   The ellipse is of type ${\mathcal P}_{23}=\{p: 0< k_1(p)< k_2(p)=k_3(p)\}$ and the hyperbole is of type   ${\mathcal P}_{12}=\{p:0< k_1(p)=k_2(p)<k_3(p)\}.$

  All  partially umbilic curves are biregular and 
  have    torsion zero  only at  the points of intersection with coordinate planes.

  \end{lemma}
  
  \begin{proof} From equations \eqref{eq:AQ3R4} and \eqref{eq:AQ3p} it follows that $k_1(p)=k_2(p)$ is defined by $v=w=c^2$ and $u\in (b^2,\infty)$ and so this partially umbilic set is the union of four open arcs.
  
   Also  $k_2(p)=k_3(p)$ is defined by $u=v= b^2$ and $w\in (d^2,c^2)$  and therefore  this partially umbilic set is the union of two closed curves.

  From the symmetries of $ Q_3$ it follows that one connected component is contained in the region $\{ (x,y,z,t):\; x>0, y=0, z>0\}.$
  
  The curves $E(v,w)=u=0$ and $H(u,w)=v=0$ stated in the lemma  are, respectively, an ellipse and a hyperbole (two connected components).
  
  Direct analysis shows that each connected  component is a biregular curve.
  \end{proof}

    \begin{lemma}\label{lem:CQ3R4} Let $\lambda \in (b^2,\infty)$ and consider the quartic surface
  $  Q_\lambda= Q_\lambda^{-1}(0)\cap  Q_3$.
  
   Then $  Q_\lambda$ has two connected components, both are diffeomorphic to an ellipsoid with three different axes  and there exists a  principal parametrization of $  Q_\lambda$  such that
      the principal lines are the coordinate
      curves. Therefore, each connected component $   Q_{\lambda}^i, \; (i=1,2)$,  of  $  Q_\lambda$ is principally equivalent to  an ellipsoid of Monge (three different axes).

 \end{lemma}
 
 \begin{proof} Similar to the proof of lemma \ref{lem:CQ1R4}.
 
 \end{proof}

    \begin{lemma}\label{lem:BQ3R4}
     Let $\lambda \in (c^2,b^2)$ and consider   the intersection of the quadric
     {\small  
     $$Q_\lambda(x,y,z,t)=\frac{x^2}{a^2+\lambda}-\frac{y^2}{b^2-\lambda}-
     \frac{z^2}{c^2-\lambda}-\frac{t^2}{d^2-\lambda}=1,\;\; a>0,\; b>c>d>0,$$
     }
 with the quadric $Q_0=Q_0^{-1}(0)$.
 Let $ Q_\lambda= Q_\lambda^{-1}(0)\cap  Q_3$.
       Then $ Q_\lambda$ has two connected components, both are diffeomorphic to 
       a
        hyperboloid of one sheet with three different axes  and there exists a  principal parametrization of $  Q_\lambda$  such that
      the principal lines are the coordinates curves. Therefore, each connected component of  $  Q_\lambda$ is principally equivalent to
      a hyperboloid of one sheet.
  
  \end{lemma}
  \begin{proof} Similar to the proof of lemma \ref{lem:BQ1R4}.\end{proof}

    \begin{lemma}\label{lem:DQ3R4}
     Let $\lambda \in (d^2,c^2)$ and consider   the intersection of the quadric
     {\small  
     $$Q_\lambda(x,y,z,t)=\frac{x^2}{a^2+\lambda}-\frac{y^2}{b^2-\lambda}-
     \frac{z^2}{c^2-\lambda}-\frac{t^2}{d^2-\lambda}=1,\;  a>0,\; b>c>d>0,$$
     }
 with the quadric $Q_3=Q_0^{-1}(0)$.
 Let $ Q_\lambda= Q_\lambda^{-1}(0)\cap Q_3$.
         Then $  Q_\lambda$ has two connected components, both are diffeomorphic to a hyperboloid of two  sheets with three different axes  and there exists a  principal parametrization of $  Q_\lambda$  such that
        the principal lines are the coordinates curves. Therefore, each connected component of  $  Q_\lambda$ is principally equivalent to a hyperboloid of two sheets.
 
   \end{lemma}
   
   \begin{proof} Similar to the proof of lemma \ref{lem:DQ1R4}.
   \end{proof}

    The results above of this section are summarized in the following theorem.
    
  \begin{theorem} \label{th:Q3R4}
  Each connected component of $  Q_3$.
  has   empty  umbilic set 
  and its partially umbilic set  consists of three curves
  ${\mathcal P}_{23}^1, \; {\mathcal P}_{23}^2,\; {\mathcal P}_{12}$.  
 They are    defined in the chart $(u,v,w)$ by:
 \begin{equation}\label{eq:puq3t}
    \aligned
    E(v,w)=& {\frac { \left( {a}^{2}+{c}^{2} \right) {v}^{2}}{{b}^{2}-{c}^{2}}}+{
    \frac { \left( {a}^{2}+{d}^{2} \right) {w}^{2}}{{b}^{2}-{d}^{2}}}-1,\;\; u=0\\
   H(u,w)=&-{\frac { \left( {a}^{2}+{b}^{2} \right) {u}^{2}}{{b}^{2}-{c}^{2}}}+{
   \frac { \left( {a}^{2}+{d}^{2} \right) {w}^{2}}{{c}^{2}-{d}^{2}}}-1,
   \;\; v=0\endaligned 
    \end{equation}

    \noindent i)  \; The principal foliation  ${\mathcal F}_1  $
  is singular on ${\mathcal P}_{12} \;$ ${\mathcal F}_3 $
  is singular on ${\mathcal P}_{23}^1\cup  {\mathcal P}_{23}^2$  and  ${\mathcal F}_2  $
  is singular on $  {\mathcal P}_{12}  \cup {\mathcal P}_{23}^1\cup  {\mathcal P}_{23}^2$.

  The partially umbilic curves ${\mathcal P}_{23}^1$ and $  {\mathcal P}_{23}^2$
  are of type  $D_1$, its  partially umbilic 
  separatrix surfaces 
    span an open band  $ W_{23}$
  such that $\partial W_{ 23 }=  {\mathcal P}_{23}^1\cup  {\mathcal P}_{23}^2$.
  
  \noindent  All the leaves of    ${\mathcal F}_1  $    are arcs and diffeomorphic to  $\mathbb R.\;$
  All the leaves of    the principal foliation ${\mathcal F}_3$ outside the band $  W_{23}$ are compact.
   See illustration in Fig. \ref{fig:PC_q3_123},  top.
   
      \noindent ii)  \;  The principal foliation ${\mathcal F}_2   $ is singular at
  ${\mathcal P}_{12} \cup     {\mathcal P}_{23}^1\cup  {\mathcal P}_{23}^2$   which are partially  umbilic 
   separatrix surfaces. The umbilic separatrices $W_{23}^1$ and $W_{23}^2$ of $ {\mathcal P}_{23}^1$ and $ {\mathcal P}_{23}^2$  are 
   non bounded punctured open disks. The umbilic separatrix $W_{12}$ of  ${\mathcal P}_{12} $ is diffeomorphic to a unitary disk with  two points removed. See Fig. \ref{fig:lq3sep}.

  \noindent All the leaves of ${\mathcal F}_2$,  outside the partially umbilic surfaces separatrices, are closed.
   See illustration in Fig. \ref{fig:PC_q3_123},  bottom.

      Moreover,  $W_{23}^1$ and $W_{23}^2$ intersect $ W_{12}$ transversally along  open arcs.  See Fig. \ref{fig:lq3sep}.

  \end{theorem}

  \begin{figure}[ht]
   \begin{center}
    \def\svgwidth{0.5\textwidth}
     \includegraphics[scale=0.18]{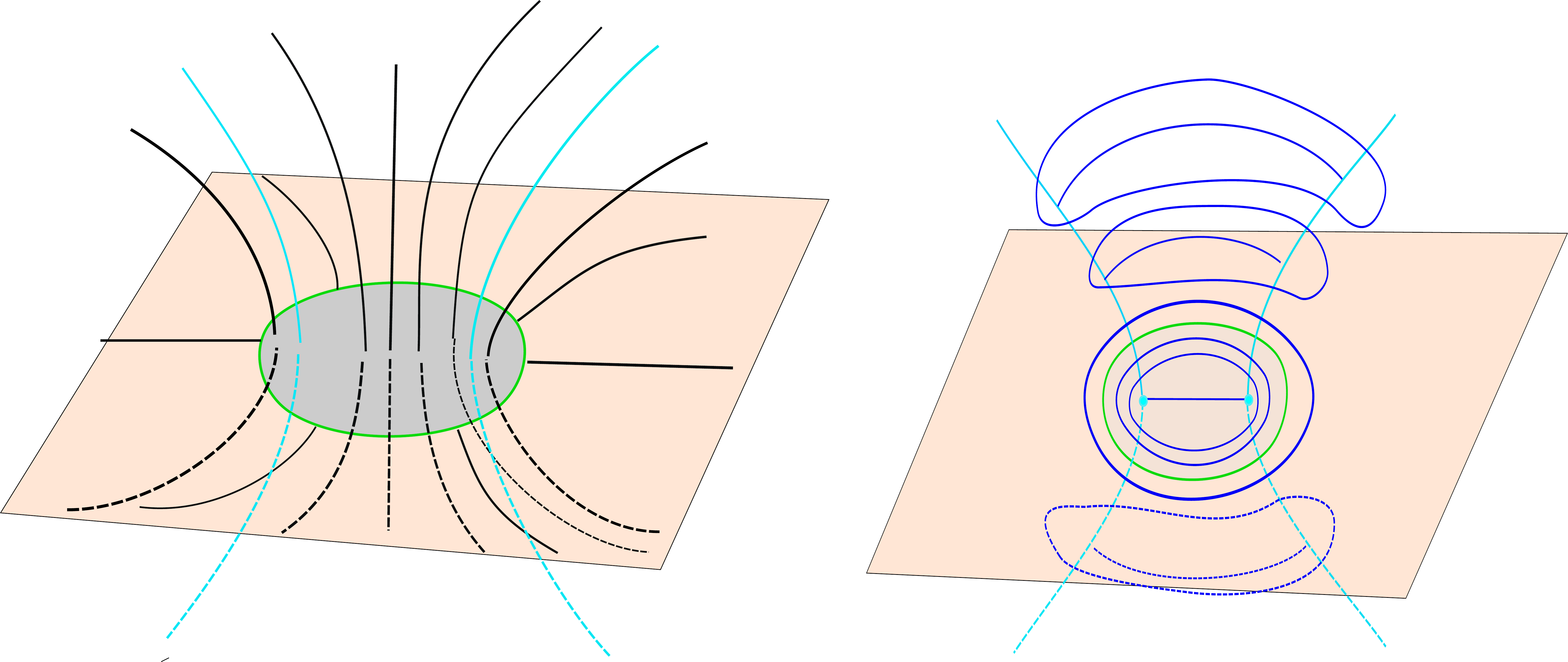}
        \def\svgwidth{0.35\textwidth}
            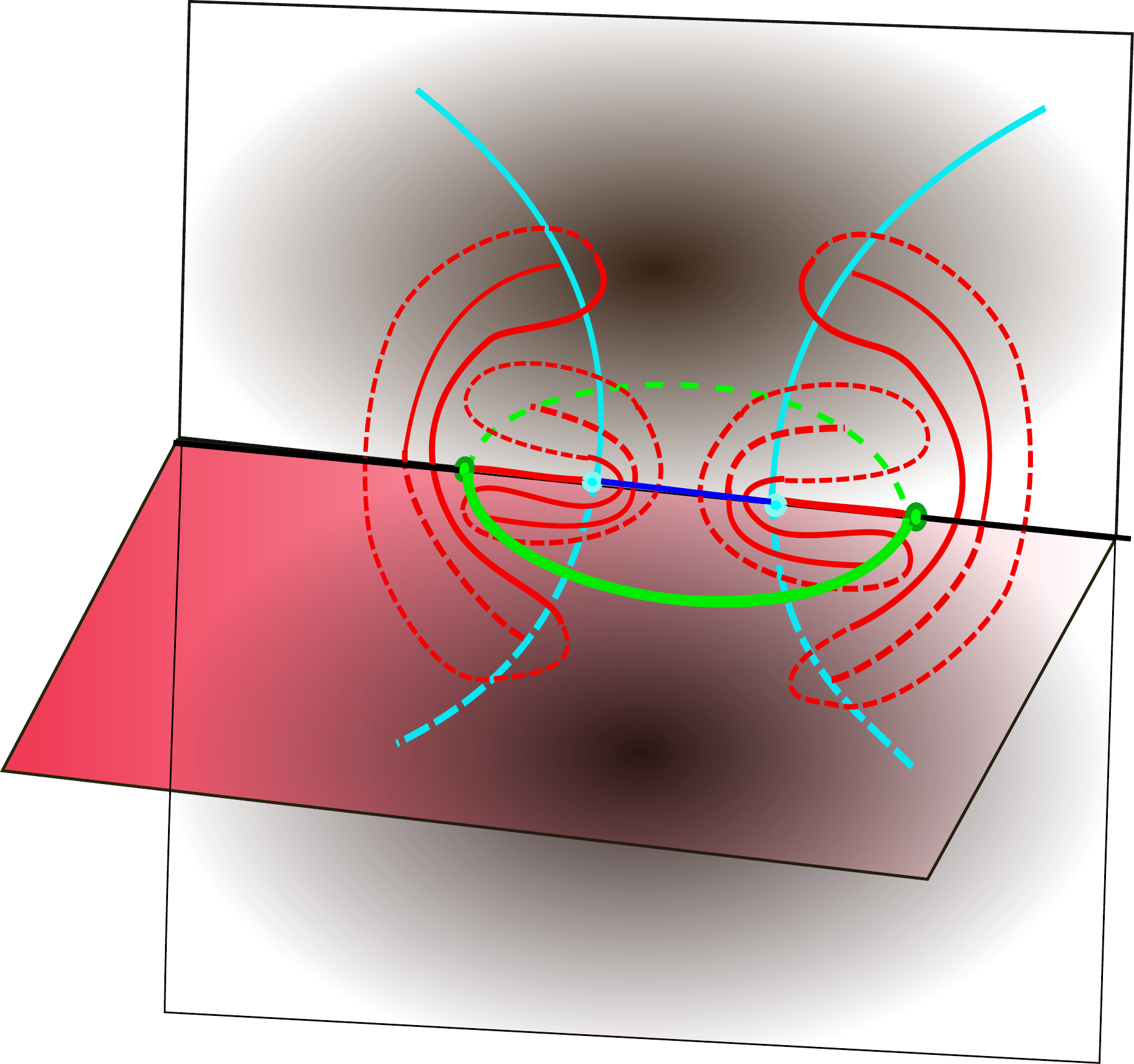
   \caption{Principal Foliations and Partially  Umbilic Curves of  $ Q_3$.
   Top 
   level: minimal  (black) and maximal (blue) curvature foliations. 
   Bottom level:  intermediate  (red)  curvature foliation.
   \label{fig:PC_q3_123} }
   \end{center}
   \end{figure}

           \begin{figure}[h]
             \begin{center}
              \def\svgwidth{0.40\textwidth}
                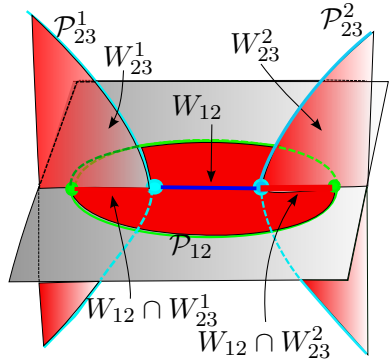
             \caption{Umbilic separatrix surfaces $W_{12},$ $ W_{23}^1$ and $W_{23}^2$ and their intersections along  open arcs.\label{fig:lq3sep} }
             \end{center}
             \end{figure}
 
  \section{Concluding Comments}\label{ss:CoCo}
  The main results presented in this paper,  sinthesized  in 
 theorems \ref{th:eabcd}, \ref{th:Q1R4}, \ref{th:Q2R4} and  \ref{th:Q3R4},  
  establish the principal configurations in all  generic quadric hypersurfaces in $\mathbb R^4$. This complements, in the generic context,  the study of quadric ellipsoids in \cite{bsbm14}, reviewed in section \ref{ss:q0abcd}. 
  
  The authors point out that case of the quadric $Q_3$ has no parallel  in the current literature. It provides the first natural algebraic example of a transversal intersection of partially umbilic separatrix  surfaces,
  known to be generic in  principal configurations in  smooth hypersurfaces of   $\mathbb R^4$ \cite{garcia-tese}. 
 
 \section*{Acknowlegment}
 The authors were supported
  by Pronex/FAPEG/CNPq Proc. 2012 10 26 7000 803. 
 \newpage
 \bibliographystyle{plain}

 \newpage
{

\vskip 0.7cm
\author{\noindent Jorge Sotomayor\\ Instituto de Matem\'atica e Estat\'{\i}stica \\
Universidade  de S\~ao Paulo\\
 Rua do Mat\~ao  1010,
Cidade Universit\'aria, CEP 05508-090,\\
S\~ao Paulo, S. P, Brazil}
 \email{sotp@ime.usp.br}

 \vskip 0.7cm

 \author{\noindent Ronaldo Garcia\\Instituto de Matem\'atica e Estat\'{\i}stica \\
Universidade Federal de Goi\'as\\ CEP 74001--970, Caixa Postal 131 \\
Goi\^ania, Goi\'as, Brazil}
 \email{ragarcia@ufg.br}
}
  
\end{document}